\documentclass[a4paper,12pt]{amsart}

\usepackage{amsthm}
\usepackage{amsmath}
\usepackage{amsfonts}
\usepackage{amssymb}
\usepackage{times}
\usepackage{graphicx}

\theoremstyle{definition}
\newtheorem{defn}{\indent\bf Definition}
\newtheorem{rem}[defn]{\indent\bf Remark}

\theoremstyle{plain}
\newtheorem{lemma}[defn]{\indent\bf Lemma}
\newtheorem{prop}[defn]{\indent\bf Proposition}
\newtheorem{thm}[defn]{\indent\bf Theorem}
\newtheorem{cor}[defn]{\indent\bf Corollary}
\newtheorem*{theorem*}{\indent\bf Theorem}
\newtheorem*{cor*}{\indent\bf Corollary}

\begin{document}

\title[Operator Algebra content of the Ramanujan-Petersson Problem]{The Operator Algebra content of the Ramanujan-Petersson Problem}
\author[Florin R\u adulescu]{Florin R\u adulescu${}^*$
  \\ \\
Dipartimento di Matematica\\ Universita degli Studi di Roma ``Tor Vergata'}
\dedicatory{Dedicated to Professor Henri Moscovici on the occasion of his 70th anniversary}

\thispagestyle{empty}

\def\tilde{\widetilde}
\def\a{\alpha}
\def\T{\theta}
\def\PSL{\mathop{\rm PSL}\nolimits}
\def\SL{\mathop{\rm SL}\nolimits}
\def\PGL{\mathop{\rm PGL}}
\def\Per{\mathop{\rm Per}}
\def\GL{\mathop{\rm GL}}
\def\Out{\mathop{\rm Out}}
\def\Int{\mathop{\rm Int}}
\def\Aut{\mathop{\rm Aut}}
\def\ind{\mathop{\rm ind}}
\def\card{\mathop{\rm card}}
\def\d{{\rm d}}
\def\Z{\mathbb Z}
\def\R{\mathbb R}
\def\cR{{\mathcal R}}
\def\tPsi{{\tilde{\Psi}}}
\def\Q{\mathbb Q}
\def\N{\mathbb N}
\def\C{\mathbb C}
\def\bH{\mathbb H}
\def\Y{{\mathcal Y}}
\def\L{{\mathcal L}}
\def\G{{\mathcal G}}
\def\U{{\mathcal U}}
\def\H{{\mathcal H}}
\def\I{{\mathcal I}}
\def\A{{\mathcal A}}
\def\S{{\mathcal S}}
\def\O{{\mathcal O}}
\def\V{{\mathcal V}}
\def\D{{\mathcal D}}
\def\B{{\mathcal B}}
\def\K{{\mathcal K}}
\def\cC{{\mathcal C}}
\def\cR{{\mathcal R}}
\def\cX{{\mathcal X}}
\def\cM{{\mathcal M}}
\def\lnu{{\mathcal L(G \rtimes L^{\infty}(\G,\mu))}}
\def\gop{{G^{\rm op}}}
\def\lkm{{\mathcal L(\Gamma \rtimes L^{\infty}(K,\mu))}}
\def\ptimes{\mathop{\boxtimes}\limits}
\def\potimes{\mathop{\otimes}\limits}

\begin{abstract}

Let $G$ be a discrete countable group, and let $\Gamma$ be an almost normal subgroup. In this paper we investigate the classification  of     (projective) unitary representations $\pi$ of $G$ into the unitary group of the Hilbert space $l^2(\Gamma)$ that extend the left regular representation of $\Gamma$. Representations with this property are obtained by restricting to $G$  square integrable representations of a larger semisimple Lie group $\overline G$, containing $G$ as dense subgroup and such that $\Gamma$ is a lattice in $\overline G$. This type of  unitary representations of of $G$ appear in the study of automorphic forms.

We prove that the Ramanujan-Petersson problem regarding the action of the Hecke algebra on the Hilbert space of $\Gamma$-invariant vectors for the unitary representation $\pi\otimes \overline \pi$ is an intrinsic problem on the outer automorphism group of the von Neumann algebra
$\mathcal L(G \rtimes L^{\infty}(\mathcal G,\mu))$, where $\mathcal G$ is the Schlichting completion of $G$ and $\mu $ is the canonical Haar measure on $\mathcal G$.



\end{abstract}

\maketitle

\renewcommand{\thefootnote}{}
\footnotetext{${}^*$ Member of the Institute of  Mathematics ``S. Stoilow" of the Romanian Academy}
\footnotetext{${}^*$
Florin R\u adulescu is supported in part by PRIN-MIUR and by a Grant of the Romanian National Authority
for Scientific Research, project number PN-II-ID-PCE-2012-4-0201.}
\footnotetext{${}^*$  email address radulesc@mat.uniroma2.it}

\section{Introduction, definitions, and main results}\label{intro}

Let $G$ be a discrete group and let $\Gamma$ be an almost normal subgroup.
In this paper we investigate the classification, up to unitary equivalence, of (projective)  unitary representations $\pi$ of $G$ on Hilbert spaces $H_\pi$,    with the property  that $\pi$ restricted to $\Gamma$  is unitarily equivalent  to the left regular representation of $\Gamma$.

 Such representations appear in the study of automorphic forms or Maass forms,  where one studies the associated vector spaces of  $\pi(\Gamma)$-invariant vectors (\cite{ma}, \cite {sa}, \cite {de}). The space of
 $\pi(\Gamma)$-invariant vectors, briefly referred to as the space of \emph{$\Gamma$-invariant vectors} in this paper, is naturally a Hilbert space. In general this is not a subspace of the original Hilbert space; it is related to the original Hilbert space if one considers a rigged Hilbert space structure (\cite{ge}) on $H_\pi$ (see e. g. the construction in \cite{Ra5}).

  The scalar product  on the space of $\Gamma$- invariant vectors corresponds to the Petersson scalar product in the case of automorphic forms (\cite{pe}). If the restriction of $\pi$ to $\Gamma$ admits a subspace $L$ such that the translates $\pi(\gamma)L, \gamma\in L$ are mutually orthogonal and generate  $H_\pi$, then the $\Gamma$-invariant vectors are naturally identified to $L$. This is not the most general case as there are situations (see bellow) when $\pi|_\Gamma$ is a multiple  of the left regular representation without admitting a subspace as above- the obstruction being the non-integrality of the Murray-von Neumann dimension of the von Neumann algebra generated by $\pi(\Gamma)$ (\cite{vN},  \cite{ghj}, \cite{AS}).

   One constructs, given such a unitary representation $\pi$ a new unitary representation $\tilde{\pi}$ of $G$, admitting proper $\Gamma$-invariant vectors, which extends to a profinite completion of $G$. The matrix coefficients of this new representation, corresponding to the subspace of $\Gamma$ invariant vectors, are the usual Hecke operators (\cite {bor}, \cite{hall}, \cite {Ra5}).

The primary aim of this paper is to reformulate in an operator algebra setting the classical problem of the Ramanujan-Petersson estimates for the Hecke operators. Putting the Ramanujan-Petersson problem into the operator algebra
context introduced in this paper has the advantage that the Hilbert
spaces acted by the unitary representations determining the Hecke
operators are naturally identified with the Hilbert algebras (\cite{Ta})
associated to the corresponding type II factors. This is obvious when considering the representation $\pi\otimes \overline{\pi}$. In this case the
$\Gamma$-invariant vectors are canonically identified with the type II$_1$ factor commutant $\pi(\Gamma)'$ (see e.g \cite{jo}) and the scalar product is determined by the $L^2$-space associated to the trace. This corresponds to the analysis of Hecke operators on Maass forms (\cite{Ra1},\cite{Ra3}).  We note that  different, but related approaches, are also considered in  the papers \cite{com}, \cite{ccm}, \cite{cho}.

 The unitary representations $\tilde{\pi\otimes \overline{\pi}}$, associated to the unitary representations $\pi\otimes \overline{\pi}$ considered above, are
induced by automorphism groups of the corresponding type II factors.
The additional structure, on the underlying unitary representations,
 is used to determine  canonical   tensor product
decompositions of such representations (see e.g. Corollary
\ref{faracoordinate}).

The data from a representation $\pi$ as above, which is used to construct Hecke operators, is transformed into data coming from a canonical unitary representation of the Schlichting completion of $G$. Splittings of canonical representations associated to $G$ are proven to be in one-to-one correspondence with unitary representations $\pi$, up to the obvious equivalence relation on the space of such representations.

The main source of such unitary (projective) representations arises from the square summable unitary representations $\tilde\pi$
of a semisimple Lie group $\overline{G}$, containing $G$ as a dense subgroup and so that $\Gamma$ is a lattice in $\overline{G}$ (see \cite{ghj}, Section 3.3),
into the unitary group $\mathcal U(H_\pi)$ of a Hilbert space $H_\pi$.
The restriction $\pi=\tilde{\pi}\vert_G$ has the property that $\pi\vert_\Gamma$ is a multiple of the
left regular representation, with multiplicity
$$D_\pi={\rm dim}_{\{\tilde\pi(\Gamma)\}''}(H_\pi)\in(0,\infty)$$
given by the Murray-von Neumann dimension of the
Hilbert space $H_\pi$ as left module over the von Neumann algebra generated by the image of the group representation. The dimension is proportional by
$\operatorname{Vol} (\overline{G}/\Gamma)^{-1}$ to the Plancherel coefficient of the square integrable representation $\tilde\pi$ (see \cite{ghj}, Section 3.3). In general the dimension $D_\pi$ is not always an integer (see e.g \cite{ghj},  \cite{Ra2} for more details on the non-integer case).

The classification of the unitary representations $\pi$ of the group $G$ with $D_\pi =1$ is analyzed in detail in the first sections.
This  corresponds to the case where $\pi|_\Gamma$ is unitarily equivalent to the left regular representation
(skewed by cocycle, if we start with a projective unitary representations).

In the last section we also consider the situation $D_\pi\ne 1$.  Although this is  similar to the previous case, it requires  additional formalism.
In  Theorem \ref{abstractsetting} (i) and formula (\ref{mostgeneral}), we construct the unitary representation
of $G$ that plays the essential role in the analysis of $D_\pi=1$ (cf. Theorem \ref{split} (ii)) in the equivalent description of a unitary representation $\pi$ as above.

For simplicity we assume that the groups $\Gamma, G$ have infinite, nontrivial conjugacy classes (the i.c.c. property \cite{Sa}), and hence that
all associated von Neumann algebras are type II$_1$ factors.
For a discrete group $N$ we denote by $\mathcal L(N)$ the associated von Neumann algebra (\cite{vN}, \cite{Sa}).

We prove that the classification problem of unitary  representations $\pi$ of $G$ as above, up to the  equivalence relation corresponding to  conjugation by unitary operators,
is determined by the first cohomology group
\begin{equation}\label{h1}
H^1_{\alpha}(G, \mathcal U(P))
\end{equation}
of the group $G$, with values in the unitary group $\mathcal U(P)$ of a type II$_1$ factor $P$, associated with a homomorphism $\alpha :G\rightarrow \operatorname{Aut}(P)$.

Both objects $P$ and $\alpha$ are constructed canonically,  starting from  a fixed unitary (projective) representation $\pi$ of $G$ as above,
through a process similar to an infinite, simultaneous, Jones's  basic construction (\cite{jo}), which we describe below (see formulae (\ref{p}) and (\ref{alpha11})).

For $\sigma$ in $G$, let
$$\Gamma_{\sigma} = \sigma\Gamma\sigma^{-1}\cap\Gamma$$
be the finite index subgroup of $\Gamma$ associated to $\sigma$.
Let $\S$ be the downward directed lattice of finite index subgroups of $\Gamma$, generated by the subgroups $\Gamma_{\sigma}$, $\sigma \in G$. We assume that $\S$ separates the elements of $\Gamma$. For simplicity, for reasons concerning the construction of the tower of commutant algebras below, we will assume that all  the groups in $\S$ have the i.c.c. property.

Let $K$ be the  profinite completion of the subgroup $\Gamma$ with respect to the family of finite index subgroups $\S$, and let $\mu$ be the corresponding Haar measure on $K$ (\cite{Sch}, \cite{Tz}).
Then, the type II$_1$ factor
\begin{equation}\label{p}
P = {\mathcal L (\Gamma \rtimes L^{\infty}(K,\mu))},
 \end{equation}
entering in the definition of the cohomology group describing the classification of unitary representations $\pi$ as  above, is
the reduced crossed product von Neumann algebra (\cite{vN},\cite {Op}) associated to the probability measure preserving action of $\Gamma$, by left translations on $(K,\mu)$.

 %
 %

To construct the homomorphism $\alpha :G\rightarrow \operatorname{Aut}(P)$,  we identify the type II$_1$ factor $P$ with the output  at infinity in the simultaneous, infinite, iterated Jones's basic construction associated to all the inclusions
\begin{equation}\label{inc}
 \mathcal L(\Gamma_0)\subseteq \mathcal  L(\Gamma), \quad \Gamma_0\in\S.
\end{equation}
The homomorphism $\alpha$ is obtained  from a homomorphism  of $G$ into the automorphism group of a larger von Neumann algebra associated to an infinite, simultaneous Jones's basic construction for pairs of isomorphic subfactors (see Proposition \ref {doublecross}).

We represent  the terms of Jones's basic construction in formula (\ref{inc}) using commutant algebras.   Since we are working in the case
$D_\pi=1$, we can     use  the canonical (anti-)isomorphism of the algebra $\mathcal L(\Gamma)$ (see e.g \cite{Co},\cite{Sa}) with the commutant
 algebra $\pi (\Gamma)'$.


Since the commutant von Neumann algebras
$\pi(\Gamma_{\sigma})'$   are type II$_1$ factors, it follows that, for $\Gamma_0\subseteq \Gamma_1$, $\Gamma_0,\Gamma_1\in \S$, the embeddings
$$\pi(\Gamma_1)'\subseteq\pi(\Gamma_0)'$$
are trace preserving.

Consider the type II$_1$ factor obtained, by completion, from the inductive limit of the  type II$_1$ factors $\pi(\Gamma_{0})'$,  $\Gamma_0 \in \S$:
 \begin{equation} \label{ainfty}
  \A_\infty=\overline{\mathop\bigcup\limits_{
   \Gamma_{0}\in \S}  \pi(\Gamma_{0})^\prime}.
\end{equation}
In the above formula the closure is taken with respect to the weak topology, in the GNS construction associated to the unique trace on the tower of algebras.

Then $\A_\infty$ is a II$_1$ factor, anti-isomorphic to  the simultaneous Jones basic construction for all the inclusions
$$\pi(\Gamma_{\sigma})'' \subseteq\pi( \Gamma)'',\quad \sigma \in \Gamma.$$

\noindent
For $\sigma\in G$, let $K_\sigma$ be the closed subgroup of $K$ defined by the formula
\begin{equation}\label{ksigma}
K_\sigma
=K\cap \sigma K\sigma^{-1}.
\end{equation}
Obviously $K_\sigma$  is the closure, in the profinite completion, of the group  $\Gamma_{\sigma}$.

We use the letter $\mathcal R$ to denote right convolutors. Then
 $$\A_\infty\cong \mathcal R(\Gamma \rtimes L^{\infty}(K,\mu)).$$
In the above isomorphism, the Jones's projection $e_{\Gamma_\sigma}$, associated to  an inclusion as in formula (\ref{inc}), corresponds to the characteristic function
$$\chi_{K_\sigma}=\chi_{\overline{\Gamma_{\sigma}}}\in L^{\infty}(K,\mu).$$
By using the canonical anti-isomorphism between the algebras of left and respectively right convolutors (\cite{Co}), we obtain that the algebra $\A_\infty$, the  inductive limit of II$_1$ factors in formula (\ref{ainfty}), is anti-isomorphic to the type II$_1$ factor $\lkm.$
Thus, in the sequel, we let
$$\A_\infty= \mathcal L(\Gamma \rtimes L^{\infty}(K,\mu))=P.$$

Using the above anti-isomorphism, the type II$_1$ factor $\pi(\Gamma_{\sigma})'$ is identified with the type II$_1$ factor
\begin{equation} \label{asigma}
\mathcal A_{\sigma \Gamma}=\L(\Gamma  \rtimes l^\infty(\Gamma/ \Gamma_\sigma))\cong \{ \mathcal L(\Gamma),\chi_{K_\sigma}\}''\subseteq  \A_\infty.
\end{equation}
Clearly, using a Pimsner-Popa basis ([PP]) for the inclusion $\L(\Gamma_\sigma)\subseteq \mathcal L(\Gamma)$, we obtain that
\begin{equation}\label{asigmaform}
P \cong\L(\Gamma_\sigma\rtimes L^\infty(K_\sigma,\mu) ) \otimes B(\ell^2(\Gamma/ \Gamma_\sigma)).
\end{equation}


\noindent Denote the Hilbert space of the unitary representation $\pi$ by $H_{\pi}$.    We consider  the homomorphism  ${\rm Ad\, }\pi$ of $G$ into the inner automorphism group of  the algebra $B(H_{\pi})$.
We observe that   the    automorphism ${\rm Ad\, }\pi(\sigma)$ leaves invariant the upward directed union of commutants
   $\mathop{\bigcup}\limits_{\Gamma_{\sigma}\in \S} \pi(\Gamma_{\sigma})'$.
    Indeed,
    for $\sigma\in G$, we have that
    $${\rm Ad\, }\pi(\sigma) [\pi(\Gamma_{\sigma^{-1}})'']=
    \pi(\Gamma_{\sigma})''.$$
    Hence
     $${\rm Ad\, }\pi(\sigma)[\pi(\Gamma_{\sigma^{-1}})']= \pi(\Gamma_{\sigma})',$$
and therefore
$${\rm Ad\, }\pi(\sigma)[\pi(\Gamma_{\sigma^{-1}} \cap \Gamma_0)']=
   \pi(\Gamma_{\sigma}\cap\sigma\Gamma_0\sigma^{-1})', \quad \Gamma_0\in\S, \sigma \in G.$$

Consequently, the unitary action ${\rm Ad\, }\pi$ induces a homomorphism $\alpha:G\rightarrow{\rm Aut}(\A_\infty)$, obtained by requiring that
     \begin{equation}\label{alpha11}
     \alpha|_{\pi(\Gamma_{\sigma^{-1}} \cap \Gamma_0)'}=
     {\rm Ad\, }\pi(\sigma)|_{\pi(\Gamma_{\sigma^{-1}} \cap \Gamma_0)'},\quad\Gamma_0\in \S,\sigma\in G.
     \end{equation}

    \noindent Using the above mentioned anti-isomorphism between the type II$_1$ factors
 $\mathcal R(\Gamma \rtimes L^{\infty}(K,\mu))$ and $P=\lkm$, we construct from formula (\ref{alpha11})   a canonical  homomorphism
 $\alpha : G\rightarrow \operatorname{Aut}(P)$ which
defines the classifying cohomology group $H^1_{\alpha}(G, \mathcal U(P))$.

We construct below a canonical  subgroup of $\operatorname{Out}(P)$. The unitary representations
 $\pi$ as above are, up to unitary equivalence,   in one-to-one correspondence with the liftings of the above mentioned  subgroup to
  $\operatorname{Aut}(P)$ (see e.g. \cite{Co} for the relevant definitions concerning groups of automorphisms of von Neumann algebras).
\begin{defn}\label{outing}
For $\sigma\in G$, let $\theta_\sigma:\Gamma_{\sigma^{-1}}\rightarrow\Gamma_\sigma$ be the group isomorphism induced by $\sigma$. Using this isomorphism and formula (\ref{asigmaform}),  it follows that any bijection between $\Gamma/\Gamma_{\sigma^{-1}}$ and $\Gamma/\Gamma_{\sigma}$ induces an isomorphism
 \begin{equation}\label{thetas}
 \tilde {\theta_\sigma}\in {\rm Aut}(P),
 \end{equation}
 with $P$ as in \eqref{p}.
Then the isomorphisms $\tilde {\theta_\sigma}$ define a map
 \begin{equation}\label{outer}
  \Phi:G\rightarrow  {\rm Out}(P)={\rm Aut(P)}/ {\rm Int}(P).
  \end{equation}
  This map is canonical, and   it is independent on  the choice of a Pimsner-Popa basis made in formula (\ref{asigmaform}).
 \end{defn}

The classification of the representations $\pi$, up to unitary equivalence, is  equivalent  to the classification of the liftings of the map  $\Phi$ to
$\operatorname{Aut}(P)$.
In particular, if there exists a unitary representation $\pi$ as above, then $\Phi$ is a group homomorphism.  In Section \ref{unicitatea} we will prove the following:

 \begin{thm}\label{liftout}
Suppose $D_\pi=1$ and let $\Phi$ be as in \eqref{outer}. Then:
 \item {(i)} Any unitary representation $\pi$ of $G$ as above defines a lifting
$$\alpha_\pi =(\alpha_\pi (g))_{g\in G}$$
of $\Phi$ to a group homomorphism $\alpha_{\pi}: G \rightarrow \operatorname{Aut}(P)$, such that:
   \begin{itemize}
  \item[(a)]
For every $\sigma \in G$, $\alpha_\pi(\sigma)$ maps $\mathcal A_{\sigma^{-1} \Gamma}$ onto
 $\mathcal A_{\sigma\Gamma}$ (see ( \ref {asigma})).
  \item[(b)] $\alpha_\pi|_\Gamma={\rm \ Id}_P .$
  \end{itemize}
\item{(ii)} For  any two representations  $\pi,\pi'$ of $G$ as above, the corresponding liftings  $\alpha_\pi$, $\alpha_{\pi'}$ are cocycle conjugated.
\item{(iii)}
Two  unitary representations  $\pi$ and $\pi'$ of $G$ are unitarily equivalent if and only if the corresponding liftings $\alpha_\pi$, $\alpha_{\pi'}$ are conjugated by an automorphism in $\operatorname{Int} (P)$.
\end{thm}

The above statement proves that if such a representation $\pi$ exists, then the canonical obstruction (\cite{Co}, \cite{Jo}, \cite{CJ})
in $H^2(G,\mathcal U(P))$, associated to the homomorphism $\Phi$ into Out$(P)$, vanishes.

Let $\G$ be the locally compact, totally disconnected group obtained as the Schlichting completion of $G$ with respect to the subgroups in $\S$.
Consider the crossed product von Neumann algebra
\begin{equation}\label{defM}
{\mathcal M}=\mathcal L(G \rtimes L^{\infty}(\G,\mu)).
\end{equation}
Let $G^{\rm op}$ be the group $G$ with opposite multiplication.
Let $\beta : G \to \Aut(\cM)$ be the canonical homomorphism of $\gop$ into the automorphism group of $\cM$, which acts  by leaving $\mathcal L(G)$ invariant and acts
by composition with right translation on $L^{\infty}(\G, \mu)$.

To prove Theorem \ref{liftout}   we show that unitary representations $\pi$ as in the introduction are in one-to-one correspondence with tensor product splittings of the canonical right action $\beta =(\beta_g)_{g\in G}$. The case $D_\pi\ne1$ is outlined in Section \ref{appendix}.

This proves that the homomorphisms $\alpha$ as above are tensor product factors of the homomorphism $\beta$.  This is contained in following theorem and its corollary,  which give equivalent descriptions for representations $\pi$ as above. Then $G^{\rm op}$  acts canonically on $\mathcal{M}$ as follows: $G^{\rm op}$ leaves
$\mathcal L(G)$ invariant and acts by left multiplication on the infinite measure space $\G$.

To state the next result, which will be proved in Sections \ref{thesplitting11} and \ref{proof1to3}, we introduce the operator system
\begin{equation}\label{os}
\S\O=\S\O(\Gamma, G) = \C(G/\Gamma) \potimes_{\C(\Gamma\setminus G/\Gamma)} \C(\Gamma \setminus G).
\end{equation}

\begin{thm}\label{split}
The following statements  are  equivalent:

\item{(i)} There exists a (projective) unitary representation $\pi : G \to \U (H_{\pi})$ such that $\pi\vert_\Gamma$ is unitarily equivalent to the left regular representation of $\Gamma$.

\item{(ii)} There exists a representation $t : \S\O \to \mathcal L(G)$ of the operator system $\S\O$, verifying the identities in the formulae
(\ref{id1}), (\ref{id2}), (\ref{id3}).

\item{(iii)} There exists a unitary representation $u$ of the group $G^{\rm op}$ into
$\mathcal L((G \times G^{\rm op}) \rtimes L^{\infty}(K, \mu))$, of the form
\begin{equation}\label{uform}
u(\sigma)=\chi_K(X^{\Gamma\sigma\Gamma} \otimes \sigma^{-1})\chi_K,\quad \sigma \in G,
\end{equation}
 where, for each double coset $\Gamma\sigma \Gamma$, $\sigma \in G$, the element $X^{\Gamma\sigma\Gamma}$ is a selfadjoint element in $\mathcal L(G) \cap l^2(\Gamma\sigma\Gamma)$.


\item{(iv)}
There exists a
 matrix unit $(v_{\Gamma\sigma_1, \Gamma\sigma_2})_{\Gamma\sigma_1, \Gamma\sigma_2 \in \Gamma \setminus G}\subseteq \mathcal M $
such that
 \begin {enumerate}
 \item[(a)] $v_{\Gamma\sigma, \Gamma\sigma} = \chi_{\overline{\Gamma\sigma}} \in L^{\infty}(\G, \mu),\quad \sigma \in G,$
\item[(b)] $\beta_g(v_{\Gamma\sigma_1, \Gamma\sigma_2}) = v_{\Gamma\sigma_1 g, \Gamma\sigma_2 g},\quad g \in G^{\rm op}, \Gamma\sigma_1, \Gamma\sigma_2\in \Gamma\backslash G .$
\end{enumerate}
\end{thm}

Part (iii) in the above theorem is an abstract  C$^\ast$-algebraic point of view for the representation $\pi$.
A unitary representation as in part (iii) may be constructed directly also in the case $D_\pi\ne 1$ (see Theorem \ref{abstractsetting}, Section \ref{appendix}).

The homomorphism $\alpha=\alpha_\pi$ acts on  the  corner $\chi_K \mathcal M\chi_K$ of the algebra
$\mathcal M$ introduced in formula \eqref{defM}.

We denote by  $\rho_{\Gamma\backslash G} $  the right quasi- regular representation of $G$ into the unitary group associated with the Hilbert space $\ell^2(\Gamma\backslash G)$. Then $\operatorname{Ad} \rho_{\Gamma\backslash G}$ is a homomorphism from $G$ into the inner automorphism group of $B(\ell^2(\Gamma\backslash G))$.

The following corollary will  be proved in Section \ref{thesplitting11}.
 \begin{cor}\label{defalpha}
We assume that the equivalent conditions (i)-(iv) in the statement of Theorem \ref{split} hold true.
Recall that  $\chi_K\mathcal{M}\chi_K$ is isomorphic to  $P$.
  The $\gop$-equivariant matrix unit constructed in part (iv) of the preceding statement, yields  a   homomorphism
$$\gop \ni g \rightarrow \alpha_g \in \Aut(\chi_K\mathcal{M}\chi_K). $$

Then:

\item {(i)} The algebra ${\mathcal M}$ splits as
   \begin{equation}\label{thesplitting}
\mathcal{M}\cong \chi_K \mathcal{M} \chi_{K} \otimes B(l^2(\Gamma \setminus G)).
\end{equation}

\item {(ii)}
The homomorphism $(\beta_g)_{ g\in G^{\rm op}}$ into the automorphism group of  $\mathcal M$, splits in the tensor product form
 \begin{equation}\label{thesplitting1}
 \alpha_g \otimes {\rm Ad\, }\rho_{\Gamma \setminus G}(g), \quad g \in G^{\rm op}.
 \end{equation}
\end{cor}

It may be easily observed that both homomorphisms $\alpha$ and $\beta$ extend to a representation of $\G$.

The following result, which will be proved in Section \ref{reltohecke}, summarizes the relation of the above construction with the Hecke operators acting on the Hilbert space
of $\Gamma$-invariant vectors
for the representation $\pi \otimes \overline{\pi}\cong $ Ad $\pi$.
Consider the canonical von Neumann conditional expectation (see e.g. \cite{Sa})
$$E_{\mathcal L(\Gamma)}^{\lkm}(\ \cdot\ ),$$
defined  on $P=\lkm$, with values onto $\L(\Gamma)$.

  \begin{thm}\label{heckes}
Assume that   the four equivalent properties in the above  theorem  hold true.
Consider  the $\ast$-algebra representation for the Hecke algebra $\H_0$ constructed in \cite{Ra1}, acting on the Hilbert space of $\Gamma$-invariant vectors for the unitary representation $\pi \otimes \overline {\pi}\cong {\rm Ad}(\pi)$.

Then the above $\ast$-algebra representation of the algebra $\H_0$   is unitarily equivalent to the representation of the Hecke algebra obtained  by  using  the completely positive maps $\Psi_{\Gamma\sigma\Gamma}$ on $\mathcal L(\Gamma)$, defined by the formula
  \begin{equation}\label{formhecke}
\Psi_{\Gamma\sigma\Gamma}(x)=E_{\mathcal L(\Gamma)}^{\lkm}(\alpha_g(x)), \quad g \in \Gamma \sigma\Gamma,
x\in \L(\Gamma).
 \end{equation}

\end{thm}

In particular, the statement of the Ramanujan-Petersson conjectures for the Hecke algebra representation into the Hilbert space of $\Gamma$-invariant vectors for the representation
$\pi \otimes \overline{\pi}\cong $ Ad $\pi$, is equivalent to the weak containment of the unitary representation of $G$ induced by
$\alpha$ on $L^2(P)\ominus \mathbb C1$, into the restriction to $G$ of the left regular representation of the Schlichting completion $\G$ of $G$ (see e. g. \cite{Lu}).

  Assume that
  $\pi$ belongs to the discrete series of PSL$(2,\mathbb R)$,  We consider the restriction  $\pi|_{{\rm PGL}(2,\mathbb Z[1/p])}$, where $p$ a prime number.  Then, by \cite{Ra1}, the     Hecke operators associated to this representation, defined as above, are unitarily equivalent to the Hecke operators acting on Maass forms.

  It is also possible to obtain a "coordinate free" version of the above statement.  In this formulation we substitute the homomorphism $\alpha$ by a canonical homomorphism. The drawback is that the type II$_1$ factor on which it acts is defined only up to an isomorphism.

  Recall that $\mathcal M$ is the reduced  crossed product von Neumann algebra $\L(G\rtimes L^\infty (\G,\mu))$, with semifinite trace $T$ induced by the $G$-invariant measure $\mu$.
Recall that  the homomorphism $\beta:G \rightarrow {\rm Aut}(\mathcal M)$, where
 each automorphism $\beta_g$ acts by right translation by $g$ on
 $L^\infty (\G, \mu)$, and acts identically on $\mathcal L(G)$, for $g\in G$.

 We also consider the algebra $D\subseteq L^\infty (\G, \mu)\subseteq  \mathcal M$ consisting of left $K$-invariant functions in $L^\infty (\G, \mu)$. Obviously $D\cong \ell^\infty(\Gamma\backslash G)$.
We have (see Section \ref{reltohecke} for the proof):

  \begin{cor}\label{faracoordinate}
Assume that the equivalent conditions    introduced above hold true. Then:
\item{(i)}
There exist a unitary representation
$\theta: G\rightarrow \U(\mathcal M)$ with the following properties:

 (a) The representation $\theta$ is unitarily equivalent to the right quasi-regular representation $\rho_{\Gamma\backslash G}$ of $G$.

 (b) For all $g,h\in G$ we have   ${\rm Ad}(\theta_g)|_D= \beta_g|_D$ and $\beta_g(\theta_h)=\theta (ghg^{-1}).$

 \item{(ii)} Let $\Theta \subseteq \mathcal M$ be the type $I_\infty$ factor generated by the image of the representation $\theta$.
 Let $\A=\Theta'\subseteq D'$. Obviously $\L(\Gamma)\subseteq \A$. Then $\A$   is a type II$_1$ factor isomorphic to $\L(\Gamma \rtimes K)$ with unique trace $\tau$ induced by $\mu$.
   \item{(iii)} The homomorphism $(\beta_g)_{g\in G}$ invariates $\Theta$ and hence it invariates $\A$.
 The homomorphism $ (\beta_g|_\A)_{g\in G}$ of $G^{\operatorname{op}}$ into ${\rm Aut} (\A)$ is unitarily equivalent to the homomorphism $\alpha$ introduced in the above corollary (Corollary \ref{defalpha}).
 It extends obviously to $\G$.

  In particular, the Hecke algebra representation introduced in the previous  theorem  is unitarily equivalent to the representation of the Hecke algebra obtained by using the  completely positive maps defined by the formula
 $$\Psi^0_{[\Gamma\sigma \Gamma]}(x)=E^{\A}_{\L(\Gamma)}(\beta_\sigma(x)),\quad x \in \A, \sigma\in G. $$

 \item {(iv)} The classification  of the unitary representations $\theta$ as in (i) is  determined by  the cohomology group $H^1_{\beta|_\A}(G,\U(\A))$.
\end{cor}

Thus,   the unitary representation
 of $G$ (extending to $\G$) determining the Hecke operators acting on the $\Gamma$-invariant vectors for the representation $\pi\otimes \overline\pi$, is an operator algebra object. The procedure to construct this operator algebra object is as follows. Consider the homomorphism $\beta:G^{\rm op}\rightarrow {\rm Aut}({\mathcal M})$ described above, acting by right translation on $\G$ and trivially on $\L(G)$.
There exists  a type I$_\infty$ factor $\Theta \subseteq {\mathcal M}$ such that $\beta_g|_\Theta$ is unitarily equivalent to the
homomorphism  ${\rm Ad}\rho_{\Gamma\backslash G}$ of $G$ into $B(\ell^2(\Gamma\backslash G))$, induced by the right quasi-regular representation and such that both $\beta_g$ and ${\rm Ad}\rho_{\Gamma\backslash G}$ act identically on the diagonal algebra $D=\ell^\infty (\Gamma\backslash G)$, and ${\rm Ad}\rho_{\Gamma\backslash G}(g^{-1})\circ \beta_g$ is a homomorphism
for every $g\in G$.

Then every $\beta_g$ invariates $\A=\Theta'$ and the homomorphism  $(\beta_g|_\A)_{g \in G}$ into
${\rm Aut}(\A)$ is conjugated to the homomorphism
$\alpha$ into Aut$(\mathcal L(\Gamma\rtimes K))$.

The  goal for Ramanujan-Petersson problem is to prove that the
unitary representation of $G$, induced by $\beta|_\A$ on
the orthogonal  complement of the scalars in $L^2(\A,\tau)$, is weakly contained in the restriction to $G$ of the left regular representation of $\G$. This property obviously holds true for the unrestricted, unitary representation induced by the homomorphism  $\beta$ on $L^2(\mathcal M,T)$. The reason why this property doesn't straightforwardly pass to a subalgebra is the fact that the two traces
defining the Hilbert spaces $L^2(\A,\tau)$ and $L^2(\mathcal M,T)$  are different: one is finite and the other one is semifinite.



\section{Outline of the paper}

To prove the classification statement, we prove first a stronger result, which identifies the group homomorphism $\alpha$, introduced in formula (\ref{alpha11}), with a tensor product factor in a
representation of the group $G$ into the automorphism group of a larger II$_\infty$ factor. This II$_\infty$ factor encodes the information on the ergodic action of the group $G$
on its Schlichting completion (\cite{Sch}, \cite{Tz}).

We  recall that $\G$, the
 Schlichting completion of $G$, is a locally compact, totally disconnected  group, obtained as the disjoint union of all double cosets $K\sigma K$, where $\sigma$ runs over the set of representatives of double cosets $\Gamma\sigma\Gamma$ of $\Gamma$ in $G$.
 Then $G$ is dense in $\G$ and $K$ is a maximal compact subgroup.

 By Jones's index theory,  the existence of a unitary representation $\pi$ as above, automatically implies  the equality of the indices
 $$[\Gamma:\Gamma_\sigma]=[\Gamma:\Gamma_{\sigma^{-1}}], \quad \sigma \in G.$$
The Haar measure $\mu$ on $K$ extends to the Haar measure on $\G$. This    measure  is also denoted by $\mu$. It  is normalized by the condition that $\mu(K)=1$.
The above equality of the indices implies that the measure  $\mu$ on $\G$ is $G$-bivariant on $\G$.

 Let $\mathcal{M}$ be the reduced crossed product von Neumann algebra defined in \eqref{defM}, with respect to the Haar measure on $\mathcal G$ (\cite{vN}).
 Let $G^{\rm op}$ be the group $G$ with opposite multiplication. Then $G^{\rm op}$  acts canonically on $\mathcal{M}$ as follows: $G^{\rm op}$ leaves
$\mathcal L(G)$ invariant and acts by left multiplication on the infinite measure space $\G$.

We establish a correspondence between unitary representations $\pi$ as above   and ${G^{\rm op}}$-equivariant splittings of the form
\begin{equation}\label{formspl}
 \mathcal M\cong P \otimes B(l^2(\Gamma \backslash G)),
\end{equation}
  of the crossed product von Neumann algebra $\mathcal{M}$.
In this identification the type II$_1$ factor $P$, introduced in formula (\ref{p}), is the corner algebra $\chi_K\mathcal M\chi_K$ of $\mathcal M$ with unit $\chi_K$, the characteristic function of $K$.

   The homomorphism $\alpha$ is then identified, in this setting, with the action, on the  tensor product   factor  component $\chi_K\mathcal M\chi_K$, in the representation from formula (\ref {formspl}), of  the canonical    embedding  of the  group ${G^{\rm op}}$ into the automorphism group of the algebra $ \mathcal M\cong P \otimes B(l^2(\Gamma \backslash G)) $.

The problem of classification, up to unitary equivalence, of the unitary representations $\pi$ of $G$ considered above,
is therefore  equivalent to the problem of the classification of the ${G^{\rm op}}$-equivariant splittings
$$\mathcal M={\mathcal L(G \rtimes L^{\infty}(\mathcal G,\mu))}\cong \mathcal L(\Gamma\rtimes K)\otimes B(l^2(\Gamma \backslash G)).$$

\noindent    In the above $G^{\rm op}$-invariant isomorphism, the  tensor product factor $B(l^2(\Gamma \backslash G))$ is  acted by  ${\rm Ad}\rho_{\; \Gamma/G}$.

This correspondence between unitary representations and ${G^{\rm op}}$-equivariant splittings is explicitly described below.
Given $\pi$, we prove  that there exists a homomorphism $\alpha_\pi
=(\alpha_g)_{g\in G}$  of the group $ G^{\rm op}$ into the automorphism group of $P$ (unique up to cocycle perturbation), such that, $G^{\rm op}$-equivariantly, the formula (\ref{formspl}) holds true. The action of the group $ G^{\rm op}$ on the two sides of the splitting in formula
(\ref{formspl}) is described below.

The diagonal algebra
$$l^\infty( \Gamma \backslash G)\subseteq B(l^2(\Gamma \backslash G))$$
is   identified, using the characteristic functions
 $\chi_{Kg}, g\in G$ with the subalgebra of $L^{\infty}(K\backslash\mathcal G,\mu)$ consisting of left  $K$-invariant, bounded measurable functions on $\G$.
  Let $\rho_{\; \Gamma \backslash G}$ be the right quasi-regular representation of ${G^{\rm op}}$ on $l^2(\Gamma \backslash G)$.

On the left hand side of the  isomorphism in formula (\ref{formspl}), the group $G^{\rm op}$ acts trivially on $\mathcal L(G)$, and  acts by  right translations on $\mathcal G$.
On the right hand side  the action of ${G^{\rm op}}$ is defined by
$$g \to \alpha_g \otimes {\rm Ad}\rho_{\; \Gamma/G}(g),\quad  g \in G.$$

Consequently,  unitary representations $\pi$ of $G$ such that $\pi|_\Gamma$ is unitary equivalent to the left regular representation of $\Gamma$
are  in one-to-one correspondence, up to unitary equivalence, with factorizations of the action of $G^{\rm op}$ by right translations
on $\G$. The factorizations "live" in the crossed product von Neumann algebra ${\mathcal M}={\mathcal L(G \rtimes L^{\infty}(\mathcal G,\mu))}.$
In the case of a projective unitary representation, the left regular representation will be replaced by the skewed left regular representation (see  the  Ozawa's notes \cite{Ra3}),
and skewed crossed products (\cite{Bo}) will be used. In the the case $D_\pi\ne 1$, the outline of the above correspondence is sketched in Section \ref{appendix}.

The homomorphism $\alpha=(\alpha_g)_{g\in G} :G\rightarrow \operatorname{Aut}(P)$ has the property that the space of $\Gamma$-fixed vectors consists of the elements in the algebra $\mathcal L(\Gamma)$.
In fact, as it will be explained later in this section, $\alpha$ has an obvious extension to a homomorphism of $\G$.

The factorization  construction also proves that the Hecke operators,  induced by the homomorphism $\alpha$ on the space of $\Gamma$-invariant vectors,    are  the same as the Hecke operators appearing   in the Hecke  algebra representation introduced in \cite{Ra1}. This is seen as follows:

Consider the canonical von Neumann conditional expectation (see e.g. \cite{Sa}),  $$E_{\mathcal L(\Gamma)}^{\lkm}(\cdot),$$ defined  on $P=\lkm$, with values onto $\L(\Gamma)$.
Then we prove  that the Hecke algebra representation on $\L(\Gamma)$ associated to $\pi, G, \Gamma$ as in \cite{Ra1}, is the same, up to unitary equivalence,  as the representation associating to a double coset $\Gamma\sigma\Gamma$ in $G$, the completely positive map
 $$ \Psi_{\Gamma \sigma \Gamma}(x) = E_{\mathcal L(\Gamma)}^{\lkm}(\alpha_{\sigma}(x)), \  \
x\in \L(\Gamma),\sigma \in G.$$

The Hecke operators are  the "block matrix coefficient" corresponding to the space of $\Gamma$-invariant vectors  of a unitary representation $W=(W_g)_{g \in G}$  of $G$
(extending to $\G$) on a larger Hilbert space. In the case studied here,  the Hecke operators (see \cite{Ra1}), act on the space of $\Gamma$-invariant vectors of the diagonal,
unitary representation of $G$, defined by the formula  $\pi\otimes \overline {\pi}\cong {\rm Ad\ }\pi  $. The representation $W$ is then obtained from $\alpha$ by letting it act
unitarily on the Hilbert space $L^2(P,\tau)$ associated to the canonical trace on $P$. Then, the homomorphism $\alpha$, which induces $W$, yields a tensor product  factor
of the canonical  representation $G^{\rm op}$, introduced above, into the $\operatorname{Aut}(\mathcal M)$.

It is obvious to see that the homomorphism $\alpha$ of $G$ extends to a representation of the group $\G$ into the bounded linear operators on the factor $P$.
Indeed, to construct such an extension one associates to the left convolutor by the characteristic function of a group $K_\sigma, \sigma\in G$, the conditional expectation on the algebra $\mathcal A_{\Gamma\sigma}$ introduced in formula
(\ref{asigma}).

The content of the Ramanujan-Petersson conjectures for the representation $\Psi$ of the Hecke algebra (\cite{krieg}) of double cosets of $\Gamma$ in $G$ constructed above,
acting  on the orthogonal complement of  the 1-dimensional space of scalars, in the space $L^2(P)$, is then equivalent to the weak containment of the unitary
representation induced by $\alpha$ on this Hilbert space,  in the left regular representation of $\G$.



\vspace{0.5cm}

\section{Construction of the $G^{\rm op}$-equivariant splitting ${\mathcal M}\cong  {\mathcal P} \otimes B(l^2(\Gamma \backslash G))$.
Proof of Theorem \ref{split}. I }\label{thesplitting11}

This section is mainly concerned with proving part of  Theorem \ref{split} and Corollary \ref{defalpha}.
We use the matrix coefficients of the unitary representation $\pi$  of $G$, which has  the property  that $\pi|_\Gamma$ is unitary equivalent to the left regular representation of $\Gamma$. We construct directly a    $G^{\rm op}$-equivariant embedding of $B(l^2(\Gamma \backslash G))$ into $\mathcal M$. Here $B(l^2(\Gamma \backslash G))$ is  acted by ${\rm Ad}\rho_{\; \Gamma/G}$.
We prove that this  equivariant representation of $B(l^2(\Gamma \backslash G))$   is  splitting $G^{\rm op}$-equivariantly, in the sense of tensor products,  the algebra
$\mathcal M$.

We recall  first the construction of the  C$^\ast$-representation $t$ of the Hecke algebra of double cosets of $\Gamma$ in $G$ into $\mathcal L(G)$, introduced in \cite{Ra1} (see also \cite{Ra3}). This representation was  subsequently extended to arbitrary Murray von Neumann dimensions in \cite{Ra2}. The representation will be used in the construction of the $G^{\rm op}$-equivariant matrix unit corresponding to the
$G^{\rm op}$-equivariant embedding of $B(l^2(\Gamma \backslash G))$ into $\mathcal M$.

Let $\H_0=\C(\Gamma \setminus G / \Gamma)$ be  the Hecke algebra of double cosets (see e.g. \cite{bc}). We let $\H_0$ act canonically on left and respectively right cosets, by left and respectively right multiplication.
Let $\H\subseteq B(l^2(\Gamma\backslash G))$ be the uniform norm closure of $\H_0$. In the terminology introduced in \cite{bc}, the C$^\ast$-algebra $\H$ is   the reduced Hecke von Neumann algebra  associated to the inclusion $\Gamma \subseteq G$.

In \cite{Ra1} (see also \cite{Ra3} for another exposition of the construction, and see Section \ref{appendix} for another proof of the  extension of this construction to arbitrary $D_\pi$) we constructed, using the matrix coefficients of the representation $\pi$, a representation $t : \H \to \mathcal L(G)$ such that $t^{\Gamma\sigma\Gamma}=t([\Gamma\sigma\Gamma])\in l^2(\Gamma\sigma\Gamma)\cap \mathcal L(G)$.

The precise formula for the representation $t$ is as follows: let $1$ be a trace vector for $\Gamma$ in the Hilbert space $H_{\pi}$ of the representation $\pi$. For  a subset  $A$ of $G$, define
\begin{equation}\label{t}
t^A=t(A) = \mathop{\sum}\limits_{\theta \in A}\overline{\langle \pi(\theta) 1, 1 \rangle} \theta.
\end{equation}

\noindent Then
$$
[\Gamma\sigma\Gamma]\rightarrow t([\Gamma\sigma\Gamma]),\quad \sigma\in G,$$
extends to a representation of $\H$ into $\mathcal L(G)$.
This also works when the representation is projective (\cite{Ra1}, \cite{Ra3}).

We proved in \cite{Ra1} that formula (\ref{t}) implies that  the representation $t$ of $\H_0$ extends to a representation of the larger operator system
$\S\O$ defined in formula  (\ref{os}). This system  is canonically identified to the vector space
$$\C\{\sigma_1\Gamma\sigma_2 \mid \sigma_1,\sigma_2 \in G\}.$$
In \cite{Ra1} (see also \cite{Ra3}) we proved that the  representation $t$ of $\H_0$ extends to a  representation $t : \S\O \to \mathcal L(G)$.
This representation is constructed using the matrix coefficients of the representation $\pi$ as above in formula (\ref{t}).

Let  $e$ be the neutral element of $G$, viewed as the identity element of the algebra $\mathcal L(G)$.
We say that  $t$ is a   representation of the operator system $\S\O$ if  the following  sets of identities hold true:
\begin{equation}\label{id1}
t(\Gamma)= e.
\end{equation}
\begin{equation}\label{id2}
t(\sigma\Gamma)^{\ast} = t(\Gamma\sigma^{-1}),\quad \sigma \in G.
\end{equation}
\begin{equation}\label{id3}
t(\sigma_1\Gamma)t(\Gamma\sigma_2) = t(\sigma_1\Gamma\sigma_2),\quad  {\sigma_1, \sigma_2 \in G}.
\end{equation}

We proved in \cite{Ra1} (see also \cite{Ra2}, \cite{Ra4}) that there exists a one-to-one correspondence between   unitary representation $\pi$ and representations $t$
of the operator system $\S\O$ with the above properties.

We identify the cosets of $K$ with the cosets of $\Gamma$,    by taking the closure in the profinite completion.
Let $l^{\infty}(\Gamma \backslash  G)$, $l^{\infty}(G / \Gamma)$ be the algebras of bounded left, and respectively right $\Gamma$-invariant functions,
that is the algebras generated by the characteristic functions of left (respectively right) cosets of $G$ by $\Gamma$.
These algebras are identified with the subalgebras of $L^{\infty}(K\backslash\G, \mu)$ and $L^{\infty}(\G/K, \mu)$  of left, and respectively right,
$K$-invariant bounded functions on $\G$. We denote by $l^2(\Gamma \backslash G)$, $l^2(G / \Gamma)$ the corresponding Hilbert spaces, and denote by
$\rho_{\Gamma/G}$, $\lambda_{G/\Gamma}$, the corresponding quasi-regular unitary representations of $G$.

For notational simplicity, when no confusion is possible, we denote in the sequel the characteristic function $\chi_{\sigma K}$ of the coset $\sigma K=\overline{\sigma \Gamma}\subseteq \G$ by
$\chi_{\sigma \Gamma}$, for $\sigma \in G$.

In Theorem \ref{split} we prove that an alternative method  to obtain representations $t$
as above, is to  the use of  following $G^{\rm op}$-equivariant matrix unit embedded in   the crossed product algebra $\mathcal{M}$.
$$
(v_{\Gamma\sigma_1, \Gamma\sigma_2})_{\Gamma\sigma_1, \Gamma\sigma_2 \in \Gamma \setminus G} =
(\chi_{\Gamma\sigma_1}t^{\Gamma\sigma_1\sigma_2^{-1}\Gamma}\chi_{\Gamma\sigma_2})_{\Gamma\sigma_1, \Gamma\sigma_2 \in \Gamma \setminus G}.
$$
The $G^{\rm op}$-equivariance of the matrix unit  is assumed to hold true with respect to the adjoint of  the right unitary representation
$\rho _{\Gamma \setminus G}$ of $G^{\rm op}$ into $l^2(\Gamma \setminus G)$.

The existence of such a matrix unit implies  that the von Neumann algebra $\mathcal{M}$ is $G^{\rm op}$-equivariantly isomorphic to $\chi_{K}\mathcal{M}\chi_{K} \otimes B(l^2(\Gamma \backslash G))$. It also implies  the isomorphism
\begin{equation}\label{corneriso}
\chi_{K}\mathcal{M}\chi_{K} =\chi_K(\lnu )\chi_{K}\cong \lkm =P.
\end{equation}
\noindent
 This is a consequence of the fact that the projections in the family $\chi_{gK}$, where $g\Gamma$ runs over a family of coset representatives of $\Gamma$ in $G$, are a partition of unity. The unit of the  algebra $\chi_{K}\mathcal{M}\chi_{K}$ is identified with $\chi_{K}=\chi_{\overline{\Gamma}}$.

 The $G^{\rm op}$-equivariant isomorphism  $$\lnu
\cong\chi_{K}\mathcal{M}\chi_{K} \otimes B(l^2(\Gamma \backslash G))$$
holds true  with respect to   a tensor product representation of $G^{\rm op}$ into the automorphism group of  $\chi_{K}\mathcal{M}\chi_{K} \otimes B(l^2(\Gamma \backslash G))$, of the form
$$\alpha_g \otimes {\rm Ad}\rho_{\Gamma \setminus G}(g), \quad  g \in \gop.$$


Recall that $\Gamma \subseteq G$ is a pair consisting of a discrete group $G$ and an almost normal subgroup $\Gamma$, both assumed to be i.c.c.
Let $\S$ be the downward directed class of subgroups of $\Gamma$, generated by $\Gamma_{\sigma} = \sigma\Gamma\sigma^{-1}\cap \Gamma$, $\sigma \in G$.
Assume that $\S$ separates the points of $\Gamma$. Let $(K, \mu)$ be the corresponding profinite completion of $\Gamma$.
Let $(\G, \mu)$ be the Schlichting extension of $G$, as introduced in Section 2. Then $\mu$ is the Haar measure
on $\G$, normalized by the condition that $\mu(K)=1$.

We assume that for all $\sigma \in G$ the subgroups $\Gamma_{\sigma}$, $\Gamma_{\sigma^{-1}}$ have equal indices.
In particular, the Haar measure on $\G$ is bivariant. Also, we assume that $G$ acts ergodicaly on $\G$, and that all groups in $\S$ are i.c.c.  Consequently, the reduced von Neumann algebra crossed product factors
$$\mathcal{M} = \mathcal L(G \rtimes L^{\infty}(\G, \mu))\quad \mbox{\rm and} \quad P = \lkm$$
\noindent are type II$_{\infty}$ (respectively II$_1$) factors.
If  the unitary representation $\pi$ is projective with cocycle $\varepsilon$, then in the definitions of $\cM$ and $P$ we take the $\varepsilon$-skewed crossed product von Neumann algebras.

We consider the following outer action of $\gop$ on $\mathcal M$ introduced in Section \ref{intro}.
\begin{defn} \label{defbeta}
Let $\beta : G \to \Aut(\cM)$ be the canonical homomorphism of $\gop$ into the automorphism group of $\cM$, acting by leaving $\mathcal L(G)$ invariant and acting
by composition with right translation on $L^{\infty}(\G, \mu)$.
\end{defn}

\begin{rem}\label{groupoid}
    Consider the canonical  action by left and right multiplication of $G\times G^{\rm op}$ on $L^{\infty}(\G, \mu)$. To this action corresponds the reduced crossed product von Neumann algebra
    $$\mathcal L((G \times G^{\rm op}) \rtimes L^{\infty}(\G, \mu)).$$

The above action reduces to   a canonical   groupoid action of $G\times G^{\rm op}$ on $K$. This is obtained  by letting,
  for $g_1,g_2 \in G$, the domain of the transformation induced by  $(g_1, g_2)\in G\times G^{\rm op}$,  be
   $$\{k\in K| g_1kg_2^{-1}\in K\}.$$
Obviously, $G\times G^{\rm op}$ acts by measure preserving transformations on $K$. Let
$$\mathcal B_\infty=\mathcal L((G \times G^{\rm op}) \rtimes L^{\infty}(K, \mu))$$
be the corresponding reduced $C^\ast$-groupoid algebra.
It is then obvious that
   \begin{equation}\label{binfty}
\chi_K\big[ \mathcal L((G \times G^{\rm op}) \rtimes L^{\infty}(\G, \mu))\big]\chi_K\cong \big[ \mathcal L((G \times G^{\rm op}) \rtimes L^{\infty}(K, \mu))\big].
\end{equation}

\end{rem}

Theorem \ref{split} shows that the data from the representation $\pi$ is encoded in a unitary representation of a special form of $G$, in the unitary group of the algebra $\mathcal B_\infty$.
Equivalently, this corresponds to a special $G^{\rm op}$-equivariant splitting of the algebra $\mathcal{M}$.



\begin{proof}[Proof of Theorem \ref{split}]
By construction, the diagonal algebra $l^{\infty}(\Gamma \setminus G)$ is independent of the choice of  the type I$_\infty$ algebra $B(l^2(\Gamma \setminus G))$
associated to the $\gop$-equivariant matrix unit.

We consider the algebra
$ L^{\infty}(K\setminus\G, \mu)$. This algebra is the weak closure of the linear span of cosets of the form $\chi_{\overline{\Gamma{\sigma}}}$ with $\sigma \in G$.
 Then we have  the isomorphism:
 \begin{equation}\label{prescribeddiagonal}
l^{\infty}(\Gamma \setminus G)\cong L^{\infty}(K\setminus\G, \mu)\subseteq \lnu.
\end{equation}
The left coset $K\sigma$ is $\overline{\Gamma\sigma}$, where the closure is taken in $\G$ for $\sigma \in G$. A similar result holds true  for right cosets.
To simplify notation, recall that when no confusion is possible, we are denoting the characteristic function $\chi_{\sigma K}$ of the coset $\sigma K=\overline{\sigma \Gamma}$ by
$\chi_{\sigma \Gamma}$ for $\sigma \in G$.

The requirement  in part (iv)  is therefore to find a $G^{\rm op}$-equivariant copy of $B(l^2(\Gamma \setminus G))$ inside $\lnu$,
with the prescribed diagonal algebra introduced in formula (\ref{prescribeddiagonal}).

The implication  (i) $\Rightarrow$ (ii) was proved in \cite{Ra1} (see also \cite{Ra2}, \cite{Ra3}).

The converse implication (ii) $\Rightarrow$ (i) is the content of Proposition 58 in \cite{Ra1}.
For convenience of the reader, we recall this proof here in the case where the representation is unitary (for the projective case see \cite{Ra1},\cite{Ra3}).

For a subset $A$ of $G$ we use the notation
$$t^A=\sum_{\theta \in A} t(\theta)\theta.$$
\noindent Let $\sigma\in G$ and let $(s_i)$ be set of representatives for $\Gamma_{\sigma^{-1}}$ in $\Gamma$. We define
$$
\pi(\sigma) s_i = [t^{\Gamma\sigma s_i}(\sigma s_i)^{-1}]^*,\quad i=1,2,\dots,[\Gamma:\Gamma_{\sigma^{-1}}].
$$
The element on the right hand side of the above equation belongs to $\mathcal L(\Gamma)$ (\cite{Ra1}, \cite{Ra3}, \cite{Ra4}).

The fact that $\pi(\sigma)$ is a representation follows form the identity
$$t(\theta_1\theta_2)= \sum_{\gamma\in\Gamma} t(\theta_1\gamma)t(\gamma^{-1}\theta_2),\quad \theta_1,\theta_2 \in G.$$
The above  identity is in turn a consequence of the identity
$$t^{\sigma_1\Gamma}t^{\Gamma{\sigma_2}}=t^{\sigma_1\Gamma\sigma_2}.$$

To prove (iv) $\Rightarrow$ (ii) we proceed as follows. For $\sigma$ in $G$, let
\begin{equation}\label{defx}
X^{\Gamma\sigma\Gamma} = \mathop{\sum}\limits_{\Gamma\sigma_1\sigma_2^{-1}\Gamma = \Gamma\sigma\Gamma} v_{\Gamma\sigma_1, \Gamma\sigma_2},
\end{equation}
with sum over all cosets $\Gamma\sigma_1, \Gamma\sigma_2 \in \Gamma \setminus G$ such that $\Gamma\sigma_1\sigma_2^{-1}\Gamma = [\Gamma\sigma\Gamma]$.

We recall that the $G^{\rm op}$-equivariance of the matrix unit
$$(v_{\Gamma\sigma_1, \Gamma\sigma_2})_{\Gamma\sigma_1,\Gamma\sigma_2}$$
means that
$$\beta_g(v_{\Gamma\sigma_1, \Gamma\sigma_2}) = v_{\Gamma\sigma_1 g, \Gamma\sigma_2 g},\quad g\in G.$$
\noindent Hence
$$\beta_g(X^{\Gamma\sigma\Gamma}) = X^{\Gamma\sigma\Gamma},\quad g \in G^{\rm op}, \sigma \in G.$$ %
\noindent
It follows from formula (\ref{defx}) that $X^{\Gamma\sigma\Gamma}$ belongs to the algebra $ \cM^G$ of fixed points for the action of $G$ on $\cM$.
Since we assumed that $G^{\rm op}$ acts ergodicaly on $\G$, it follows that
$$\cM^G=\mathcal L(G),$$
and hence
$$X^{\Gamma\sigma\Gamma}\in \mathcal L(G),\quad \sigma \in G.$$

Obviously, since $v_{\Gamma\sigma, \Gamma\sigma}$ is equal to $\chi_{\Gamma\sigma}$, it follows that
the partial isometry $v_{\Gamma\sigma_1, \Gamma\sigma_2}$ will map the space of the projection $\chi_{\Gamma\sigma_2}$ onto $\chi_{\Gamma\sigma_1}$.
Hence, using formula (\ref{defx}) defining $X^{\Gamma\sigma\Gamma}$, it follows   that if $\Gamma\sigma_1, \Gamma\sigma_2$ are so that $$[\Gamma\sigma_1\sigma_2^{-1}\Gamma ]= [\Gamma\sigma\Gamma],$$
\noindent
 then
$$\chi_{\Gamma\sigma_1}X^{\Gamma\sigma\Gamma}\chi_{\Gamma\sigma_2} = v_{\Gamma\sigma_1, \Gamma\sigma_2},\quad  \sigma,\sigma_1, \sigma_2 \in G.$$
\noindent From formula (\ref{defx}) we also obtain that:
$$
\chi_{\Gamma\alpha}X^{\Gamma\sigma\Gamma}\chi_{\Gamma\beta} = \delta_{[\Gamma\alpha\beta^{-1}\Gamma], [\Gamma\sigma\Gamma]}v_{\Gamma\alpha, \Gamma\beta},
\quad\alpha, \beta, \sigma \in G.
$$
 Here, we use the symbol $\delta$ to denote the Kronecker symbol.

Let $\theta$ be any element  in $G$. Then the property that
$$\chi_{\Gamma\sigma_1} \theta \chi_{\Gamma\sigma_2} \neq 0,$$
\noindent is equivalent to the  existence of $\gamma_1, \gamma_2\in\Gamma$ such that
$$\theta\gamma_2\sigma_2 = \gamma_1\sigma_1.$$
This holds true if and only if $\theta\in\Gamma\sigma_1\sigma_2^{-1}\Gamma$.
Consequently  $$X^{\Gamma\sigma\Gamma}\in \L(G) \cap l^2(\Gamma\sigma\Gamma).$$

We analyze  the product of two elements $X^{\Gamma\sigma_1\Gamma}$ and $X^{\Gamma\sigma_2\Gamma}$ as in formula (\ref{defx}), corresponding to two double cosets. The product corresponds to a pairing of cosets in the product of the double cosets.   Because  $(v_{\Gamma\sigma_1, \Gamma\sigma_2})_{\Gamma\sigma_1, \Gamma\sigma_2}$ is a matrix unit, the same pairing of  cosets shows up in the product formula for the corresponding double cosets $[\Gamma\sigma_1\Gamma]$ and  $[\Gamma\sigma_2\Gamma]$ in the Hecke algebra  $\H_0$.

This proves   that  the operators $X^{\Gamma\sigma\Gamma}$, $\sigma \in G$, which are a priori affiliated to $\L(G)$, have also the property that the correspondence
\begin{equation}\label{heckecor}
 [\Gamma\sigma\Gamma] \to X^{\Gamma\sigma\Gamma},\quad \sigma \in G,
 \end{equation}
extends to a $\ast$-algebra representation of $\H_0=\C(\Gamma\setminus G /\Gamma)$.
Obviously, this is trace-preserving with respect to the traces $\tau$ on $\mathcal L(G)$ and the canonical trace $\langle \cdot [\Gamma], [\Gamma]\rangle$ on $\H_0$.
Consequently, the  $\ast$-representation in formula (\ref{heckecor}) extends to a $C^\ast$-representation of  the reduced $C^{\ast}$-Hecke algebra
$$\H = C^{\ast}_{\rm red}(\Gamma \setminus G/\Gamma)=\overline{\H_0}^{||\cdot||}\subseteq B(\ell^2(\Gamma\backslash G)).$$
\noindent
Hence the elements
$$t^{\Gamma\sigma\Gamma}=X^{\Gamma\sigma\Gamma}, \quad \sigma \in G$$
are bounded.

For all $\sigma\in G$, we denote the coefficient of $\theta \in \Gamma\sigma\Gamma$  in  $t^{\Gamma\sigma\Gamma}\in \L(G)$ by $t(\theta)$. Thus
$$t^{\Gamma\sigma\Gamma}=\sum_{\theta\in \Gamma\sigma\Gamma}t(\theta)\theta\in \L(G),\quad \sigma\in G.$$

The property that
$$
\chi_{\Gamma\alpha}t^{\Gamma\alpha\beta^{-1}\Gamma}\chi_{\Gamma\beta}t^{\Gamma\beta\gamma^{-1}\Gamma}\chi_{\Gamma\gamma} = \chi_{\Gamma\alpha}t^{\Gamma\alpha\gamma^{-1}\Gamma}\chi_{\Gamma\alpha}
$$

\noindent implies, when moving in the left side member the characteristic function $\chi_{\Gamma\beta}$ to the right (using the multiplication formula (\ref{rulem},
Section \ref{appendix})), a family  of identities of the form $$\sum t^{A_i}t^{B_i} = \sum t^{C_j}.$$
These family  of identities, when summing  over unions of cosets of subgroups in $\S$ whose unions are $\Gamma$-cosets, produces exactly  the  family of identities
$$
t^{\sigma_1\Gamma}t^{\Gamma\sigma_2} = t^{\sigma_1\Gamma\sigma_2},\quad \sigma_1, \sigma_2 \in G.
$$

\noindent These identities  are exactly the sufficient conditions that imply, as recalled in the introductory part of this section (see \cite{Ra1}),
that the map $[\sigma\Gamma] \to t^{\sigma\Gamma}$, $\sigma \in G$, extends to a representation of the operator system $\S\O$, as in property (ii) in the statement.

To prove (iii) $\Rightarrow$ (iv), we note that
\begin{equation}\label{u}
 u(\sigma) = (v_{\Gamma, \Gamma{\sigma}} \otimes 1)\otimes (1 \otimes \sigma^{-1}),\quad \sigma \in G.
 \end{equation}

\noindent Hence, using formula (\ref{uform}) it follows  that
\begin{equation}\label{vg}
v_{\Gamma, \Gamma{\sigma}} = \chi_{\Gamma}(t^{\Gamma\sigma\Gamma})\chi_{\Gamma\sigma^{-1}}
\end{equation}
is an isometry in the von Neumann algebra ${\mathcal M}$, with initial space $\chi_{\Gamma}$ and range the space of the projection $\chi_{\Gamma\sigma^{-1}}$.

We define
\begin{equation}\label{defv}
v_{\Gamma\sigma_1, \Gamma\sigma_2} = \beta_{\sigma_1}(v_{\Gamma, \Gamma\sigma_2\sigma_1^{-1}}),\quad
\Gamma\sigma_2 \in \Gamma \setminus G.
\end{equation}
\noindent
By $G^{\rm op}$-equivariance and because of formula (\ref{vg}), the expression in \eqref{defv} is equal to
$$\chi_{\Gamma\sigma_1}t^{\Gamma\sigma_1\sigma_2^{-1}\Gamma}\chi_{\Gamma\sigma_2}.$$
Again, because of the $G^{\rm op}$-equivariance, this is a partial isometry from $\chi_{\Gamma\sigma_2}$ onto $\chi_{\Gamma\sigma_1}$.

The property that the family of unitaries $(u(\sigma))_{\sigma \in G}$ is a representation of $G$ translates into the fact that the family
$$(v_{\Gamma\sigma_1, \Gamma\sigma_2})_{\Gamma\sigma_1,\Gamma\sigma_2 \in \Gamma \setminus G},$$
defined in formula \eqref{defv}, is a matrix unit. Since $G^{\rm op}$ acts on ${\mathcal M}$ by leaving $\mathcal L(G)$ invariant, and since it acts through by the Koopmann unitary  representation
by right translations on $L^{\infty}(\G, \mu)$, it follows that the above
 above matrix unit is $G^{\rm op}$-equivariant.

The proof of the implication (i) $\Rightarrow$ (iii) is postponed to Section \ref{proof1to3}.
A direct, alternative proof, valid also for $D_\pi\ne1$, will be given in Theorem \ref{abstractsetting} in Section \ref{appendix}.
\end{proof}

\begin{proof}[Proof of Corollary \ref{defalpha}]
Recall that  $\chi_K\mathcal{M}\chi_K \cong\mathcal L(\Gamma\rtimes L^\infty (K,\mu))=P.$ We let
$$P =\chi_K\mathcal{M}\chi_{K} =\chi_{\overline{\Gamma}}(\mathcal L(G \rtimes L^{\infty}(\G, \mu)))\chi_{\overline{\Gamma}} = \lkm. $$
Here the unit of the above algebra is  identified with the characteristic function $\chi_{\overline{\Gamma}}$.
The construction of the $G^{\rm op}$equivariant splitting \eqref{formspl} is straightforward, once a $G^{\rm op}$-equivariant matrix unit is given.

For $p \in P$, we define
\begin{equation}\label{alpha}
\alpha_g(p) = v_{\Gamma, \Gamma g}\beta_g(p)v_{\Gamma g, \Gamma}.
\end{equation}
This is consistent with the fact  that $\beta_g(p)$ belongs to $\chi_{\Gamma g}\mathcal{M}\chi_{\Gamma g}$, for all $p\in P$ and $g\in G$.
\end{proof}

\begin{rem}\label{uback}
 The relation between the various constructions in the  preceding statement is summarized as follows. Let $\pi$ be as above a unitary representation of $G$. Consider the representation $t$  of the operator system $\S\O$ introduced in formula (\ref{os}).  Then
\item {(i).}
  The formula for the unitary representation $u$ of $G$ into $\mathcal B_\infty$ is
$$u(\sigma)=\chi_K(t^{\Gamma\sigma\Gamma} \otimes \sigma^{-1})\chi_K,\quad \sigma\in G.$$
\item{(ii).} The formula that gives back the representation $\pi$ from $t$ (or $u$) is $$\pi(\sigma)(x) =
E_{\mathcal L(\Gamma)}^{\mathcal L(G)}(t^{\Gamma\sigma\Gamma}x\sigma^{-1}),\quad  x \in l^2(\Gamma), \sigma \in G. $$

\item{(iii).}
In the identification $\chi_K\mathcal{M}\chi_K\cong \mathcal L(\Gamma\rtimes L^\infty (K,\mu))$, the homomorphism $\alpha$ coincides with the homomorphisms $\alpha$
constructed in the introduction, using the simultaneous infinite Jones basic construction.

 \item{(iv)} The restriction $\alpha|_\Gamma$ acts as the identity operator on the subalgebra $$\mathcal L(\Gamma) \subseteq \mathcal L(\Gamma)\rtimes L^\infty (K,\mu).$$

\end{rem}

\begin{proof}
The only non-obvious property is contained in (iv). We use the notations from the proof of the previous statement.
Since $G^{\rm op}$ acts trivially on $G$, it follows that $\beta_g$ acts trivially on $\mathcal L(\Gamma)$ for $g\in G$. It follows that $\alpha_{\gamma}$ acts trivially on $\mathcal L(\Gamma)$ for $\gamma$ in $\Gamma$.
Indeed, in this case, for every $$x \in \mathcal L(\Gamma) \subseteq \chi_{\overline{\Gamma}}\lkm \chi_{\overline{\Gamma}},$$
and for every $ \gamma \in \Gamma^{\rm op}$, we have that
$$\alpha_{\gamma}(x) = v_{\Gamma, \Gamma}\beta_{\gamma}(x)v_{\Gamma, \Gamma} = x,$$
showing that $\alpha |_{\Gamma}$ acts identically on $\mathcal L(\Gamma)$.
\end{proof}

In the  proof of Theorem \ref{split} we showed that there exists a correspondence between representations $\pi$ of $G$ with the properties from (i) of the preceding
theorem and  unitary representations $u$ of $G$, having an expression as in formula (\ref{u}), with values  in the subgroup of elements of
 $\U (\B_\infty)$ that are normalize $\A_\infty$. We state this separately in the following proposition.

\begin {prop}\label{doublecross}
We use the notation and  the equivalent hypothesis from  Theorem \ref{split}. Recall that
$$\A_\infty=\lkm,\quad \B_\infty=\mathcal L((G \times G^{\rm op}) \rtimes L^\infty (K, \mu)).$$
\noindent
We embed $G$ into the first component of $G \times G^{\rm op}$, and canonically extend this embedding  to an embedding of   $\A_\infty$  into $\B_\infty$.
Let $u$ be the unitary representation of $G$ into $\U(\B_\infty)$, constructed  in formula (\ref {u}). Then

 \item{(i)}  For every $\sigma\in G$ the unitary
  $u(\sigma)$ normalizes $\A_\infty$.

\item{(ii)}
The homomorphism $\alpha$ associated to $\pi$, constructed in Corollary \ref{defalpha} and mapping $G$  into $\operatorname{Aut}(P=\A_\infty)$, is computed by the formula:
\begin{equation}\label{normalizer}
\alpha_\sigma= {\rm \ Ad\ } u(\sigma)|_{\A_\infty},\quad \sigma \in G.
\end{equation}
\item{(iii)} The homomorphism $\alpha$ extends to a homomorphism of $\G$ into $\operatorname{Aut}(\A_\infty)$.
\end{prop}

\begin{proof}
We use the notation from the previous theorem and its proof.
We recall that the formula (see formulae (\ref{u1}), (\ref{u})) for $u(\sigma)$ is
\begin{equation*}
\begin{split} u(\sigma) & =\chi_{\overline{\Gamma}}(t^{\Gamma\sigma\Gamma} \otimes 1)(1 \otimes \sigma^{-1})\chi_{\overline{\Gamma}} \\
& =\chi_{\overline{\Gamma}}(t^{\Gamma\sigma\Gamma} \otimes 1)\chi_{\overline{\Gamma\sigma}}(1 \otimes \sigma^{-1}) = (v_{\Gamma, \Gamma\sigma} \otimes 1)(1 \otimes \sigma^{-1}),
\quad \sigma \in G.
\end{split}
\end{equation*}
Thus,  for $x = \chi_{\overline{\Gamma}}x\chi_{\overline{\Gamma}}\in\lkm$, which is identified to $$ \lkm\otimes 1,$$ we have that $u(g) x u(g)^{\ast}$ is equal, with the above identification, to $$u(g)(x \otimes 1)u(g)^{\ast} = (v_{\Gamma, \Gamma\sigma} \otimes 1)\chi_{\Gamma\sigma}(1 \otimes \sigma^{-1})[(\chi_{\overline{\Gamma}}x\chi_{\overline{\Gamma}}) \otimes 1](1 \otimes \sigma)(v_{\Gamma\sigma, \Gamma} \otimes 1).$$
This is thus equal to
$$
(v_{\Gamma, \Gamma\sigma}\beta_{\sigma}(x)v_{\Gamma\sigma, \Gamma}) \otimes 1,
$$
and so the unitary $u(\sigma) \in \chi_{\overline{\Gamma}}(\mathcal L((G \times G^{\rm op}) \rtimes L^{\infty}(\G)))\chi_{\overline{\Gamma}}$ normalizes $P$
(which is identified to $P\otimes 1$) and
$$\alpha_{\sigma}(x) = {\rm Ad}u(\sigma)(x), \quad x \in P, \sigma \in G.$$
This completes the proof of properties (i) and (ii) in the statement.

To prove (iii) we proceed as follows.
For $\sigma \in G$, let  $\chi_{K_\sigma}$ be the characteristic function of the  subgroup $K_\sigma$ of $K$ introduced in formula (\ref{ksigma}).
Then  the convolutor in $C^\ast (\G)$ by the function $\mu(K_\sigma)^{-1}\chi_{K_\sigma}$ is a projection, denoted by $p_\sigma$.
Let $\mathcal A_{\Gamma\sigma}$ be the subfactor introduced in formula (\ref{asigma}). To extend $\alpha$ to a homomorphism of the locally compact,
totally disconnected group $\G$, one represents the projection $p_\sigma$  in $C^\ast(\G)$ by the conditional expectation
  $E^{\mathcal A_\infty}_{\mathcal A_{\Gamma\sigma}}$.
\end{proof}

\section{The operator algebra representation of the Hecke algebra associated to the representation $\pi\otimes \overline \pi$.
Proof of Theorem \ref{heckes}}\label{reltohecke}

In this section    we describe the relation between the construction in the previous theorem and the construction of Hecke operators in \cite{Ra1}.

We prove first  that  the Hecke operators associated to the unitary representation of G,
$$\pi\otimes \overline{\pi}\cong {\rm Ad\ } \pi,$$
constructed in \cite{Ra1}, where $\pi$ is as in  part (i) of Theorem \ref{split},  may be constructed directly using  a splitting of    the ergodic action of the group   $G \times G^{\rm op}$ on $\G$ as
in part (iv) in Theorem \ref{split}, formulae (\ref{thesplitting}) and
(\ref{thesplitting1}). Let $\alpha$ be the homomorphism of $G$ into the automorphism group of $P=\mathcal L(\Gamma\rtimes K)$, constructed in Corollary \ref{defalpha}.

\begin{proof}[Proof of Theorem \ref{heckes}]

Note that the definition of $\Psi_{\Gamma\sigma\Gamma}$ is independent of the choice of $\sigma$ in $\Gamma\sigma\Gamma$, since
  by property (iv) in Remark \ref{uback} the homomorphism  $\alpha \vert_\Gamma$ acts as   the identity on $\mathcal L(\Gamma)$.

We also note that because $\mathcal L(\Gamma)$ is the space of $\Gamma$-invariant vectors in $P=\mathcal L(\Gamma\rtimes \L^\infty(K,\mu))$ for the homomorphism $\alpha$,
it follows that formula (\ref{formhecke}) defines indeed a representation of the Hecke algebra.

Recall that by formula (\ref{alpha}) in the proof of Corollary \ref{defalpha} we have
$$\alpha_g{(p)}=v_{\Gamma, \Gamma g}\beta_g(p)v_{\Gamma g, \Gamma}, \quad  p \in P.$$

Because of formula (\ref{normalizer}), we obtain that for all $\sigma\in G$ and
$$x\in \mathcal L(\Gamma) \subseteq \chi_{\overline{\Gamma}}\lkm\chi_{\overline{\Gamma}},$$
we have
$$\alpha_{\sigma}(x)=\chi_{\overline{\Gamma}}t^{\Gamma\sigma\Gamma}\chi_{\Gamma\sigma}x\chi_{\Gamma\sigma}t^{\Gamma\sigma\Gamma}\chi_{\overline{\Gamma}}.$$
Note that the last expression depends only on the coset $\Gamma\sigma \in \Gamma \setminus G$.

In particular, for every
$$x \in \L(\Gamma)=\chi_{\overline{\Gamma}}\L(\Gamma)\chi_{\overline{\Gamma}} \subseteq \chi_{\overline{\Gamma}}\lkm\chi_{\overline{\Gamma}}=\lkm,$$
we obtain that
\begin{equation}\label{e1}
E_{\L(\Gamma)}^{\lkm}(\alpha_{\sigma}(x))=\mathop{\sum}\limits_{i}\chi_{\overline{\Gamma}}t^{\Gamma\sigma\Gamma}\chi_{\Gamma\sigma s_i}x\chi_{\Gamma\sigma s_i}t^{\Gamma\sigma\Gamma}\chi_{\overline{\Gamma}}.
\end{equation}

\noindent Here, the families $(s_i)$, $(t_i)$ are the Pimsner-Popa bases used in the proof of (i) $\Rightarrow$ (iii) in Theorem \ref{split} from Section \ref{proof1to3}.

Then, the  right hand term in the equation (\ref{e1})  is  $$\chi_{\overline{\Gamma}}t^{\Gamma\sigma\Gamma}xt^{\Gamma\sigma\Gamma}\chi_{\overline{\Gamma}}.$$
This expression  is further equal to the value at $x$ of the completely positive map  constructed in \cite{Ra1}, in correspondence with a representation $\pi$ as  in statement (i) of  the previous theorem.
\end{proof}

Consequently, the problem of determining the continuity of the representation $\Psi$ of the Hecke algebra with respect to the norm on the  C$^\ast$-reduced Hecke algebra (\cite{bc}), which is the essence of the Ramanujan-Petersson problem (\cite{Ra1}),  is reduced to the analysis of the unitary representation induced by $\alpha$ on the Hilbert space $L^2(P,\tau)\ominus \mathbb C 1$.

We recall that in the previous section, Proposition \ref{doublecross}, we introduced the following notations:
$$\A_\infty=\lkm,
 \quad \B_\infty=\mathcal L((G \times G^{\rm op}) \rtimes L^\infty (K, \mu)).$$
 We use the trivial embedding $G \times \{e\}\subseteq G \times G^{\rm op}$. We  consider the reduced crossed product von Neumann algebras
 $$\mathcal C_\infty=\mathcal L(G  \rtimes L^\infty (\G, \mu))\subseteq \D_\infty=\mathcal L((G \times G^{\rm op}) \rtimes L^\infty (\G, \mu)).$$
Note that
\begin{equation}\label{abcd}
\A_\infty=\chi_K \mathcal C_\infty \chi_K \ ; B_\infty=\chi_K\D_\infty\chi_K.
\end{equation}
We also consider the algebra $D\subseteq L^\infty (\G, \mu)\subseteq  \mathcal C_\infty$ consisting of left $K$-invariant functions in $L^\infty (\G, \mu)$. Obviously $D\cong \ell^\infty(\Gamma\backslash G)$.

By $\rho_{\Gamma\backslash G} $ we denote the right regular representation of $G$ into the unitary group associated with the Hilbert space $\ell^2(\Gamma\backslash G)$.
With $\rho=\rho_G$ we denote the right regular representation of $G$. For $g \in G$, we identify the  unitary element $\rho_g$ with the  unitary, left convolution operator in $\L(G^{\rm op})\subseteq \D_\infty \ $, that corresponds  to $g$. Then
\begin{equation}\label{defbeta1}
 \beta_g= {\rm \ Ad}(\rho_g),\quad g \in G,
 \end{equation}
defines an homomorphism from $G^{\rm op}$ into  $\rm{Aut}(\mathcal C_\infty)$.

It is obvious that every $\beta_g$ acts on $\mathcal C_\infty$ by right translation by $g$ on
 $L^\infty (\G, \mu)$. Each automorphism  $\beta_g$  acts identically on $\L(G)$.

We introduce  below a different  equivalent formulation to the four equivalent statements in Theorem \ref{split}.
This is used to describe the homomorphism $\alpha$ introduced in Corollary \ref{defalpha},
in a form that is independent of the choices made in the tensor product splitting defining $\alpha$.

\begin{lemma}\label{leftreg1}
The equivalent statements in Theorem \ref{split} are further equivalent
to the following statement:
\item {(v)} There exists a unitary representation
$\theta: G\rightarrow \U(\mathcal C_\infty)$ with the following properties

\indent {(v.1)} The unitary  representation $\theta$ of $G$ is unitarily equivalent to the representation $\rho_{\Gamma\backslash G} $.

 \indent {(v.2)} For all $g$ in $G$, the unitary element $\theta(g^{-1})\rho(g)$ belongs to
the relative commutant $D'\cap \mathcal C_\infty$. Equivalently,
$${\rm Ad} \theta_g|_D= {\rm Ad} \rho_g|_D,\quad g\in G.$$

\indent  {(v.3)} The unitaries $$W_g= \theta(g^{-1})\rho(g),\quad g\in G,$$
define a unitary representation of $G$ into the type II$_1$ von Neumann algebra
$$P_0=\{\theta (G)\}'\cong \A_\infty.$$
\end{lemma}

\begin{proof}
We use the notations from the statement and proof of Theorem \ref{split}. We prove first that the equivalent conditions (i)-(iv) imply condition (v).
 Consider the selfadjoint elements constructed in formula
(\ref{defx}). Because these elements generate a
 $G^{\rm op}$-equivariant matrix unit it follows that the formula
\begin{equation}
\label{deftheta}
\theta_g=\sum_{\Gamma\sigma\in \Gamma \backslash G}\chi_{K\sigma g}X^{\Gamma\sigma\Gamma}\chi_{K\sigma},\quad g \in G,
\end{equation}
defines a unitary representation of $G$ into $\mathcal C_\infty$, that is unitary equivalent to the left regular quasi-representation $\rho_{\Gamma\backslash G} $.

 For $g\in G$,  both Ad$\rho_g$ and  Ad $\theta_g$ normalize $D$ and induce the same action on the algebra $D$.

 It remains to prove property (v.3). Formula (\ref{deftheta}) implies that
 \begin{equation} \label{rotheta}
 \rho_h(\theta_g)\rho_{h^{-1}}=\theta (hgh^{-1}),\quad h,g \in G.
 \end{equation}

 The above formula obviously implies that
 $$\theta(h^{-1})\rho_h\in \theta(G)',\quad  h\in G,$$
and also that that $h\rightarrow \theta(h^{-1})\rho_h$ is a unitary representation of $G$. This completes the proof of the direct implication.

To prove the converse we note  that the unitary representation in part (iii) of Theorem \ref{split} may be defined by the
formula
$$u(g)=\chi_K W_g\chi_K,\quad  g \in G.$$
\end{proof}

We summarize the results from the previous lemma in the context of Corollary \ref{defalpha}

\begin{lemma}\label{launloc}
With notations introduced above we define
$$\tilde \beta_g= {\rm Ad}( W_g)|_{\mathcal C_\infty},\quad g\in G.$$
Then, we have the following, $G$-invariant isomorphism:
$$\mathcal C_\infty \cong \A_\infty \otimes B( \ell^2(\Gamma\backslash G)).$$
Using the above identification and the homomorphism
$\alpha:G\rightarrow {\rm Aut}(\A_\infty)$ introduced in Proposition
\ref{defalpha}, we obtain:
$$ {\rm Ad} (\theta _g)\beta_g=\tilde \beta_g=\alpha_g\otimes {\rm Id}_{B( \ell^2(\Gamma\backslash G))}.$$
\end{lemma}

We provide now a "coordinate free" description of the homomorphism $\alpha:G \rightarrow {\rm Aut}(\A_\infty)$.
This is important because, as also mentioned above, the Ramanujan-Petersson problem for the action of the Hecke operators  on $\Gamma$-invariant vectors for the
unitary representation $\pi\otimes \overline \pi$ is equivalent to the weak containment, of the unitary representation of $G$ induced by $\alpha$ on
$L^2(\A_\infty,\tau)\ominus \C 1$, in the restriction to $G$, of  unitary left regular representation of the Schlichting completion $\G$ of $G$.

Recall that the action of  $\beta_g$  on $\mathcal C_\infty$ is defined as follows. The automorphism $\beta_g$ acts  by right translation by $g$ on
 $L^\infty (\G, \mu)$, and it  acts identically on $\mathcal L(G)$.

\begin{proof}[Proof of Corollary \ref{faracoordinate}]
Parts (i), (ii) follow immediately from Lemma \ref{leftreg1} and its proof (see formula (\ref{rotheta})). Property (iii) follows from the representation in Lemma \ref {launloc}.
This is because the homomorphisms $(\beta_g|_\A)_{g\in G}$ and $\alpha$ are the same.
The difference between the two cases is that in the case considered in this statement we no longer perform a splitting procedure to find explicitly the factor $\A$.
On the other hand, in the present context the factor
$$\L(\Gamma)\cong \L(\Gamma\times\{e\})\subseteq \mathcal C_\infty$$
is canonically embedded in the commutant algebra $\Theta'$, which by definition is equal to $\A$.
 The formula in part (iii) is then a direct consequence of the corresponding formula in Theorem \ref{heckes}.

 To prove the uniqueness in part (iv), note that if $\theta^1$ is another representation as in (i), then necessarily there exists
 a unitary $w\in D'$ such that
 $$\theta^1_g=w\theta_gw^\ast,\quad g \in G.$$
 We write the second condition in (b) for the representation $\theta^1$.
 It follows that
 $$\beta_h(w)\beta_h(\theta_g)\beta_h(w^\ast)=w\theta_{hgh^{-1}}w^\ast, \quad  g,h\in G.$$
 Hence
 $$c(g):=w^\ast \beta_h(w) \in \Theta'=\A,\quad g \in G.$$
 Thus $c$ is defining a cocycle in the cohomology group $H^1_\beta(G,
 \U(\A))$. Its triviality would  imply that $w\A w^\ast=\A.$
\end{proof}

\begin{rem}
Let $\pi$ be a representation as in part (i) of Theorem \ref{split}.
Since $\pi|_\Gamma$ is unitary equivalent to  the left regular representation of $\Gamma$, we identify the Hilbert space of the representation $\pi$ with $\ell^2(\Gamma)$.
We  may also assume that  the representation $t$ constructed in formula (\ref{t}) has the property that
$t(\Gamma\sigma\Gamma)$ belongs to the full $C^\ast$-algebra $C^\ast(G)$ for all $\sigma$ in $G$.
This  is possible by choosing a suitable cyclic trace vector in the formula defining $t$ (see \cite{Ra1}).

Let $C^\ast((G\times G^{\rm op})\rtimes L^\infty(\mathcal G,\mu))$ be the full  crossed product $C^\ast$- algebra and let $C^\ast(G\times G^{\rm op})\rtimes L^\infty(K,\mu)$
be the full groupoid crossed product $C^\ast$- algebra. Let $\Pi$ be the obvious representation of $C^\ast(G\times G^{\rm op})\rtimes L^\infty(K,\mu)$ into  $B(\ell^2(\Gamma))$.
Then:
\item { (i)}
The formula for the representation $u$ in part (iii) of Theorem \ref{split} defines a representation $\boldsymbol U$ of the group $G$ with values in the unitary group of the
$C^\ast(G\times G^{\rm op})\rtimes L^\infty(K,\mu)$. The explicit formula for $\boldsymbol U$ is:
\begin{equation}\label{U}
\boldsymbol U(\sigma)=\chi_K(t^{\Gamma\sigma\Gamma} \otimes \sigma^{-1})\chi_K .
\end{equation}

\item {(ii)} The  correspondence $\boldsymbol\Psi$ defined by
\begin{equation}\label{P}
\Gamma\sigma\Gamma\rightarrow \chi_K(t^{\Gamma\sigma\Gamma} \otimes t^{\Gamma\sigma^{-1}\Gamma})\chi_K,\quad \sigma \in G,
\end{equation}
extends by linearity to  a unital $\ast$-representation of the Hecke algebra $\mathcal H_0$ into
$C^\ast(G\times G^{\rm op})\rtimes L^\infty(K,\mu)$.

\item {(iii)} Composing the representations $\boldsymbol U,\boldsymbol\Psi$ with representation  $\Pi$ one obtains the representation $\pi$, and respectively the representation $\Psi$,
of the Hecke algebra into $B(\ell^2(\Gamma))$ constructed in Theorem \ref{heckes}

\end{rem}
\begin{proof}
The statements { (i), (ii)} follow from the results in Section \ref{appendix}. The last statement is a direct consequence of the formula (\ref{formhecke}) and of Remark \ref{uback}.
\end{proof}




\section{The algebraization   of the  space of intertwiners   and Jones's basic construction for a pair of isomorphic factors }\label{algebraization}
We first introduce an abstract setting that serves the purpose of describing the intertwiners spaces between subalgebras of the form
$$\pi(\Gamma_0)', \quad \Gamma_0\in \S.$$
Let  $\sigma \in G$,  and let
$$\theta_\sigma:\Gamma_{\sigma^{-1}}\rightarrow\Gamma_\sigma,$$
be the group homomorphism implemented on $\Gamma_{\sigma^{-1}}$  by the adjoint map $\sigma\cdot\sigma^{-1}$ of the group element $\sigma$.
We  obviously have the intertwining property
$$\pi(\sigma)\pi(\gamma_0)=\pi(\theta_\sigma(\gamma_0))\pi(\sigma),\quad \gamma_0\in  \Gamma_{\sigma^{-1}}.$$

To describe the space of all intertwiners with properties as above, we introduce a construction similar to Jones's basic construction.
Instead of working with a single subfactor, we start with  a   pair of subfactors of equal index, with a fixed isomorphism $\theta$, mapping one subfactors onto the other one.

As in the case of Jones's simultaneous  construction for the infinite family of subgroups $\Gamma_0\in \S$, where the result was the inductive limit factor $\A_\infty$ in formula \eqref{ainfty},
we also perform  an infinite simultaneous construction for the pairs of subgroups.

The  inductive limit for the space of von Neumann algebra intertwiners between subgroups in $\S$ is the  type II$_1$ factor
  \begin{equation}\label{cp1}
  \B_\infty= \chi_{K}(\mathcal L((G \times G^{\rm op}) \rtimes L^{\infty}(\G, \mu)))\chi_{K}.
  \end{equation}
Then $\B_\infty$ contains $\A_\infty$. By construction, the II$_1$ factor $\B_\infty$  encodes the algebra structure of the spaces of intertwiners, corresponding to all subgroups in $\S$.

In the above correspondence between the space of intertwiners and the algebra $\B_\infty$,  the representation $\pi$ then corresponds to a unitary representation $u$ of $G$
into the unitary group of the normalizer of $\A_\infty$ in $\B_\infty$. We have already observed  (see Proposition \ref{doublecross}, formula (\ref{normalizer})) that
$${\rm Ad\,}u(\sigma)[ \A_\infty]= \A_\infty,\quad \sigma \in G.$$
The homomorphism $\alpha$ is obtained as the restriction to $\A_\infty$ of the adjoint representation ${\rm Ad\ }u$ of $G$:
$$\alpha(\sigma)= {\rm Ad\, }u(\sigma)|_ {\A_\infty},\quad \sigma \in G.$$

In the following lemma we present an abstract formalism which establishes a correspondence between  intertwiners,
such as $\pi(\sigma)$, $\sigma \in G$, of the algebras $\pi(\Gamma_{\sigma^{-1}})'$ and $\pi(\Gamma_{\sigma})'$, with elements in the algebra
associated to the crossed product introduced in formula (\ref{cp1}).

The essential use that we make of the abstract formalism of spaces of intertwiners, is the fact that the composition operation of two intertwiners corresponds to the product operation
in the crossed product algebra $\B_\infty$, introduced in formula (\ref{cp1}).

The above correspondence between the two product operations is used in the proof of (i) $\Rightarrow$ (iii) in Theorem \ref{split}.
It shows that the expression in formula (\ref{unitary}) defines a unitary representation of $G$.
For an alternative verification, one could also use Theorem \ref{abstractsetting} (i) in Section \ref{appendix}.

The following two statements are probably known to specialists in subfactor theory. As we did not find a reference, we state and prove them directly.

 We perform a construction which is analogous with Jones's construction of the first step in the basic construction (\cite{jo}).
 Recall that  in that case of a single factor $N\subseteq M$, the first term of the basic construction is the algebra $$\langle M, e_N \rangle = Me_N M^{\rm op}\cong N'\subseteq B(L^2(M,\tau)).$$
Here we reproduce Jones's basic construction for a pair of isomorphic subfactors of equal index.
For a pair of isomorphic subfactors $N_0, N_1$, we replace the commutant algebra $N'\cong Me_N M$ with the space of intertwiners.
This space also carries a natural Hilbert structure, which we describe below.

 The following definition and  lemma are also the subject of Appendix 1 in \cite{Ra1}.  We reproduce the proofs from that paper in the actual context.

\begin{defn} \label{jones}
Let  $M$ be a II$_1$ factor with trace $\tau$.
Let    $N_0, N_1$ be  a pair of finite index subfactors of  $M$, having the same index $[M:N_0]=[M: N_1]$. Assume we are given    a fixed  isomorphism $\theta : N_0 \to N_1$. Let $e_{N_0}, e_{N_1}$ be the corresponding Jones projections onto the subfactors $N_0, N_1$.
We introduce the following constructions:

\item{(i)}
The space of $\theta$-intertwiners is   defined by the formula:
  $${\rm Int }_{\theta}(N_0, N_1) = \{ X \in B(L^2(M, \tau)) \mid X n_0 = \theta(n_0)X, \  n_0 \in N_0 \}.$$

\item{(ii)}
Define $W_{\theta} : L^2(N_0,\tau) \to L^2(N_1,\tau)$ by
$$W_{\theta}(n_0) = \theta(n_0), \quad n_0 \in N_0.$$
Viewed as an element of $B(L^2(M, \tau))$, $W_{\theta}$ is a partial isometry mapping the projection $e_{N_0}$ onto  the projection $e_{N_1}$.
  \end{defn}

 %

In analogy with the Hilbertian $M$-bimodules $Me_{N_0} M$ and $Me_{N_1}M$, which are the Jones basic construction terms from the inclusions $N_i\subseteq M$, $i=0,1$,
we construct a Hilbertian $M$-bimodule $MW_{\theta}M$,  corresponding to the pair of isomorphic subfactors.

\begin{defn}
The Hilbertian $M$-bimodule  $MW_{\theta}M$ is obtained from the free $M$-bimodule $M \times M$, by taking the quotient corresponding to the  relations:
\item{(i)}
\begin{equation}\label{intint}
W_\theta n_0=\theta(n_0)W_\theta, \quad  n_0\in N_0,
\end{equation}
\item{(ii)}
\begin{equation}\label{en}
W_\theta= W_\theta e_{N_0}= e_{N_1} W_\theta.
\end{equation}
\item{(iii)}
The elements  $mW_\theta m'\in MW_{\theta}M$ depend only on  $me_{N_1}$ and $e_{N_0} m'$.
\item{(iv)} The following formula defines a scalar product on $MW_{\theta}M$.
\begin{equation}\label{scalar}
\langle mW_\theta m',aW_\theta a'\rangle =\tau(a^*m\theta (E_{N_0}(m'(a')^*),\quad  m,m',a,a'\in M.
\end{equation}
\end{defn}

In the previous definition, we have to justify that the compatibility of the four conditions that describe the subspace defining the quotient of the free module. This is done in the following lemma:

\begin{lemma}
We use the notations from the previous statement.
The assumption (iv) means implicitly that we are factorizing the free bimodule $M\times M$  by the vector space corresponding to vectors of zero norm with respect to the scalar product introduced in formula
(\ref{scalar}). The definition of the corresponding scalar product is compatible with the conditions (i)-(iii).
\end{lemma}

\begin{proof}
Let $m,m',a,a'\in M$, $n_1 \in N_1$.
Using the scalar product introduced in formula (\ref{scalar}), we have to prove  that
$$
\langle m n_1 W_\theta m'-m W_\theta\theta^{-1}(n_1) m', a W_\theta a'\rangle=0.
$$
\noindent We have
\begin{equation*}
\begin{split}
\langle m n_1 W_\theta m', aW_\theta a'\rangle & =
\tau(a^*mn_1\theta(E_{N_0}(m'(a')^*) \\
& =\tau(a^*m\theta(\theta^{-1}(n_1))\theta(E_{N_0}(m'(a')^*))) \\
& =\tau(a^*m\theta(\theta^{-1}(n_1))(E_{N_0}(m'(a')^*))) \\  & =
\tau(a^*m\theta(E_{N_0}(\theta^{-1}(n_1)m'(a')^*))).
\end{split}
\end{equation*}

\noindent Here we use the fact that $E_{N_0}$ is a conditional expectation
and that $\theta^{-1}(n_1)$ belongs to $N_0$.
\end{proof}

Note that the scalar product corresponds exactly to the Stinespring dilation of the
completely positive map $m\rightarrow \theta(E_{N_0}(m))$, viewed
as a map from $M$ with values into $N_1\subseteq M$.

\begin{defn}\label{l2}
 We refer to  the above definitions.
Let $V$ be any unitary operator in  ${\rm Int }_{\theta}(N_0,N_1)$.  Then
$${\rm Int }_{\theta}(N_0,N_1)=V(N_0)'=(N_1)'V.$$
Hence ${\rm Int }_{\theta}(N_0,N_1)$ carries a canonical scalar product, induced by the trace on the commutant algebra $N_0^\prime$, or equivalently by the trace on the commutant algebra $N_1^\prime$. Let
$L^2({\rm Int }_{\theta}(N_0,N_1))$ be the Hilbert space completion.
This is a Hilbertian $M$-bimodule.
\end{defn}

Using the above definition, we prove in this more general setting that the correspondence in the Jones basic construction for a subfactor inclusion $N\subseteq M$ between $N'$ and $Me_NM$,
holds true also in the case of a pair of isomorphic subfactors.
\begin{lemma}\label{oldbg1}
 Let  $L^2(MW_{\theta}M)$ be the Hilbert space completion of $MW_{\theta}M$ with respect to the  scalar product introduced in formula (\ref{scalar}).  Then:

\item{(i)}
The  Hilbertian $M$-bimodules $L^2({\rm Int }_{\theta}(N_0, N_1))$ and  $L^2(MW_{\theta}M)$  are isomorphic by the following the $M$-bimodule anti-linear   map:
\begin{equation}\nonumber
 \Phi_\theta: {\rm Int }_{\theta}(N_0,N_1) \rightarrow L^2(MW_\theta M),
 \end{equation}
    defined by the relation
\begin{equation}\label{scalar1}
\langle mW_\theta m',\Phi_\theta(X)\rangle=\tau(X(m')m), \quad   m, m'\in M.
\end{equation}
\item {(ii)}
$\Phi_\theta$ extends to an isometry on the Hilbert space completion of  ${\rm Int }_{\theta}(N_0,N_1)$, with respect to the given  scalar products.
\end{lemma}

\begin{proof}
To prove (i), by bijectivity  we  may verify instead the converse.
We have to check that, with the  definition  in formula (\ref{scalar1}),
$$X(n_0m)=\theta(n_0)X(m), n_0\in N_0,\quad  m\in M.$$
By taking a trace of a product with an element   $m'\in M$, we have to check that
$$
\tau(X(n_0m)m')=\tau(X(m)m'\theta(n_0)).
$$
By using the above definition of $\Phi_\theta(X)$ this amounts to
$$
\langle m'\sigma n_0m,\theta(X)\rangle=
\langle m'\theta (n_0)\sigma m,\theta(X)\rangle,\quad  n_0\in N, m, m'\in M.
$$
This is true from the definition of the bimodule property of $MW_\theta M$.
\end{proof}

If $N^1\subseteq N\subseteq M$ is a chain of subfactors, then the inclusion $$N'\subseteq (N^1)',$$ may be interpreted in terms of the Jones bimodules structure  of the
bimodules associated to the basic construction as an inclusion
$$M {e_{N^1}}M\subseteq Me_NM.$$
We extend this to the case of  pairs of subfactors. For simplicity, we assume that all indices are integer valued.

\begin{defn}
We use the definitions and notations assumed above.
 We restrict  the isomorphism $\theta$ to an isomorphism of a smaller pair of subfactors $N_0^1\subseteq N_0, N_1^1\subseteq N_1$ such that $\theta (N_0^1)=N^1_1$.
  Assume that the inclusions $N^1_i\subseteq N_i$ have equal, integer index.

   It is obvious that
 \begin{equation}\label{inclusion11}
{\rm Int }_{\theta}(N_0, N_1)\subseteq {\rm Int }_{\theta}(N^1_0, N^1_1).\end{equation}
Since  the index of the inclusion $N^1_0\subseteq N_0$ is an integer, we may find  a  Pimsner-Popa basis $(r_j)$ for $N^1_0\subseteq N_0$, consisting of unitaries.
Let  $e_{N_0}$, $e_{N^1_0}$ be the corresponding Jones projection.
Then $e_{N_0}$ is given by
$$e_{N_0}= \sum_j r_je_{N^1_0}r_j^\ast.$$
Hence, since
$$W_{\theta|N^1_0}= W_\theta e_{N^1_0},$$
we have a formal  inclusion
\begin{equation}\label{inclusion111}
MW_\theta M = MW_\theta e_{N_0} M\subseteq MW_\theta e_{N^1_0} M= MW_{\theta|N^1_0}M.
\end{equation}
\end{defn}

\begin{lemma}\label{oldbg2}
The maps $\Phi_\theta$ and $\Phi_{\theta|N_0^1}$ are compatible with the inclusions in formulae (\ref{inclusion11}) and
(\ref{inclusion111})

 \end{lemma}
 \begin{proof}
 To prove part (ii), we have to  check the compatibility with inclusion of the maps $\Phi$.
 In the bimodule construction, this  corresponds to replacing in the above bimodule the partial isometry $W_\theta$ by $$W_\theta e_{N_0^1}=
 e_{N_1^1}W_\theta=W_{\theta|_{N_0^1}}.$$
When using the maps $\Phi_\theta$, $\Phi_{\theta|N^1_0}$ the above
formal inclusion expresses exactly the compatibility of $\Phi_\theta$, $\Phi_{\theta|N^1_0}$ with the inclusion maps.
\end{proof}

In the following lemma, we obtain an explicit formula for the map $\Phi_\theta$ constructed above. As in the previous statements we assume that all subfactors have integer index.
This condition is probably not necessary, but we consider it to have a simpler form of the statements.

 \begin{lemma}\label{oldbg3}
In the context introduced above, let $k=[M:N_0]$ and
assume that $(s_j)_{j=1}^k$ is a left Pimsner-Popa orthonormal basis for $N_0$ in $M$.
Then $\Phi_\theta$ has the following formula:
\begin{equation}\label{pp}
\Phi_\theta(X)=\sum_j (X(s_j))^*W_\theta s_j, \quad  X \in {\rm Int }_{\theta}(N_0, N_1).
\end{equation}
\end {lemma}

\begin{proof}


We note that the decomposition
$$MW_\theta M^{\text{op}}=\bigcup_j[MW_\theta s_j]$$
is orthogonal. Fix $X\in {\rm Int }_{\theta}(N_0, N_1).$ By the orthogonality property, we may assume that
$$\Phi_\theta(X)=\sum_j x_jW_\theta s_j.$$

By hypothesis $M$ is, as left $N_0$-bimodule, the ($N_0$-orthogonal) sum of $N_0 s_j$. Then
$$X(n_0s_j)=\theta(n_0)X(s_j), \quad n_0 \in N_0.$$

Denote $t_j=X(s_j)$. Then $(t_j)$ is a Pimsner-Popa orthonormal basis for $N_1$ in $M$.

The relation between $\Phi_\theta(X)$ and $X$ is:
$$
\langle m_0W_\theta m_1,\Phi_\theta(X)\rangle= \tau(X(m_1),m_0).
$$
Hence
$$
\langle X(m_1), m_0\rangle=
\langle m_0^*W_\theta m_1,\Phi_\theta(X)\rangle.
$$
Hence, for a fixed $j$, taking $m_1=s_j$ we obtain:
$$
\langle t_j, m_0\rangle=
\langle X(s_j),m_0\rangle=\langle m_0^*W_\theta s_j,\Phi_\theta(X)\rangle=
\langle m_0^*W_\theta s_j,x_iW_\theta s_j\rangle.
$$

Hence, it follows that for all $m_0\in M$ we have
$$
\langle t_j, m_0\rangle=
\langle m_0^*,x_j\rangle,
$$
or that $\tau(t_j m_0^*)=\tau(x_j^*m_0^*)$
and hence $t_j=x_j^*$, leading to
$$
\Phi_\theta(X)=\sum_j  t_j^*W_\theta s_j.
$$
\end{proof}

We use the previous construction to transform the unitary representation $\pi$ of $G$ considered in part (i) of Theorem \ref{split}, into a unitary representation of $G$ with values in the algebra $\mathcal B_\infty$.
The next lemma translates   for the product formula for intertwiners in the context of spaces of bimodules.

For simplicity, we assume from now on in this section that the factor $M$ is a group von Neumann  algebra (eventually skewed by a cocycle),
and that all subfactors that are considered correspond to  subgroups of finite index.
We also assume that all  isomorphisms between subfactors are induced by subgroup isomorphisms (this condition may certainly be relaxed, but in our main application this condition  is  verified anyway).

\begin{lemma}\label{product}
Consider two   pairs of equal index subfactors $N_0,N_1$ and $N_1,N_2$ as above, and isomorphisms $\theta_0:N_0\rightarrow N_1$ and $\theta_1: N_1\rightarrow N_2$.
Assume that there exists a pair $(N^1_0, N^1_2)$ of equal finite  index subfactors of $M$, $N^1_0\subseteq N_0$, $N^1_2\subseteq N_2$. Assume also that there exists an isomorphism  $\theta^1_0$
mapping $N^1_0$ onto $N^1_2$, such that the composition $\theta_1\circ\theta_0$ is defined on $N^1_0$ and it is equal to $\theta^1_0$.
Then:

\item{(i)}
The composition operation
\begin{equation}\label{composition}
{\rm Int}_{\theta_0}(N_0,N_1)\times  {\rm Int}_{\theta_1}(N_1,N_2)\rightarrow
{\rm Int}_{\theta^1_0}(N^1_0,N^1_2),
\end{equation}
is well-defined.
\item{\item(ii)}
The   product operation in formula (\ref{composition})  induces a    product map
\begin{equation}\label{compbimodules}
MW_{\theta_0}M\times MW_{\theta_1}M\rightarrow MW_{\theta^1_0}M,
\end{equation}
compatible with the the maps $\Phi_{\theta_0}, \Phi_{\theta_1}, \Phi_{\theta^1_0} $ constructed in Lemma \ref{oldbg1},  corresponding to the composition of  $\theta_1$ with $\theta_0$.

The image of $MW_{\theta_0}M\times MW_{\theta_1}M$ in \eqref{compbimodules} is  a sub-bimodule of $MW_{\theta^1_0}M$.
\end{lemma}

\begin{proof} The only verification that needs checking  is the fact the product formula on products of bimodules corresponds to the composition of intertwiners. To verify the product formula, by linearity it is sufficient to consider simple intertwiners that map a set of coset representatives into a set of coset representatives. In this case the product formula is just a consequence of the enumeration of cosets
  for subgroups and of subgroups of subgroups  \end{proof}

\begin{rem}\label{binftyinductive}
More generally,  one may consider a family, indexed by a set $S$, of pairs of equal index subfactors, along with isomorphisms mapping the first subfactors into the second:
$$N_0^s \mathop{\to}\limits^{\theta_s} N_1^s,\quad s\in S.$$
By the technical assumption, the Jones projections for the subfactors $N_0^s, N_1^s$ mutually commute for all $s\in S$.

We assume that for any two elements $\theta_{s_1}, \theta_{s_2}$ in the family $(\theta_s)_{s\in S}$, there exists
a third element $\theta_{s_0}$ in the family such that the composition $\theta_{s_1}\circ  \theta_{s_2}$ restricts to $\theta_{s_0}$.

Let $\B_\infty$ be the union of all bimodules in the family. Then
 $\B_\infty$ has the associative $\ast$-algebra structure introduced in formula (\ref{compbimodules}) .
 We define a trace on $\B_\infty$ by composing the trace on $M$ with the projection onto the $M$-th component.

By performing the above  construction, we obtain  a von Neumann algebra
$\B_\infty$, the simultaneous Jones's basic construction for the family of  pairs of subfactors ($N_0^s \mathop{\to}\limits^{\theta_s} N_1^s)_{s\in S}$.
Then $\B_\infty$ is isomorphic to the algebra of all intertwiners in the family $(\theta_s)_{s\in S}$, containing all  intertwiners
$W_{\theta_s}, s\in S$, all the corresponding Jones projections, and the algebra $M$.
\end{rem}

\section{Proof of the implication (i) $\Rightarrow$ (iii) in Theorem \ref{split}}\label{proof1to3}

  In this section we complete the proof of Proposition \ref{split} by using the construction from Section \ref{algebraization}.
  A different proof follows from part (i) in Theorem \ref{abstractsetting}. The latter proof has the advantage that it also works in the case $D_\pi\ne 1$.
  On the other hand the realization of spaces of  intertwiners in terms of bimodules, as in this section, is used in Section \ref{unicitatea}, Proposition \ref{cociclicitate} (i), formula (\ref{vvv}),
  to find a  compact formula for cocycles for unitary perturbations of unitary representations as in part (iii) of Theorem \ref{split}.

\begin {proof}[Proof of the implication (i)$\Rightarrow$ (iii) in Theorem \ref{split}]

We identify the commutant algebra $\pi(\Gamma)'$ with $\mathcal R(\Gamma)$, and
Jones's simultaneous basic construction for the subfactors $\pi(\Gamma_\sigma)'' \subseteq \pi(\Gamma)''$ with the inductive limit of II$_1$ factors:
$$\mathop\bigcup\limits_{\{e\} \leftarrow \Gamma_{\sigma}, \Gamma_{\sigma}\in \G}  \pi(\Gamma_{\sigma})'.$$
The II$_1$ factor $\A_{\infty}$ is the inductive limit of the II$_1$ factors
$$\pi(\Gamma_\sigma)' = \pi(\Gamma)' \bigvee \{ e_{\Gamma_{\sigma}} \}'', \quad \sigma \in G,$$
where $e_{\Gamma_{\sigma}}$ are the Jones projections (\cite{jo}) corresponding to the subfactor inclusion
$$\pi(\Gamma_\sigma)'' \subseteq \pi(\Gamma)'',\quad \sigma\in G.$$
Then $\A_{\infty}$ is  identified with $\mathcal R(\Gamma \rtimes L^\infty(K,\mu))$, where $K$ is the profinite completion of $\Gamma$
with respect to the subgroups $\Gamma_{\sigma}$. The isomorphism  is realized by identifying, for each $\sigma\in G$, the Jones projection  $e_{\Gamma_{\sigma}}$ with the characteristic function
$$\chi_{\overline{\Gamma_{\sigma}}} \in C(K)$$
of the closure of the subgroup $\Gamma_{\sigma}$,  in the profinite completion $K$. As in the previous sections, we employ the canonical anti-isomorphism (\cite{Sa}) to identify $\A_{\infty}$ with
$P=\lkm.$

To construct  the unitary representation $u$ of $G$ into
$$\chi_{\overline{\Gamma}}(\mathcal L((G \times G^{\rm op}) \rtimes L^{\infty}(\G, \mu)))\chi_{\overline{\Gamma}},$$
we use Lemma \ref{jones} and Lemma \ref{product} .

For every $\sigma\in G$ we consider the subgroup $\Gamma_{\sigma} \subseteq \Gamma$ and let the isomorphism
$$\theta_{\sigma}:\mathcal L(\Gamma_{\sigma^{-1}}) \to \mathcal L(\Gamma_{\sigma})$$
be  the conjugation by $\sigma$. By taking inductive limit as in Remark \ref{binftyinductive}, we obtain
a type II$_1$ factor $\B_{\infty}$. We denote by $W_{\theta_\sigma}$ the isometry  $W_{\sigma}= {\rm Ad\ }\sigma $ restricted to  $\ell^2(\Gamma_{\sigma^{-1}})$.

The bimodules constructed  in Lemma \ref{jones} are of the form:
$$\mathcal L(\Gamma)(W_{\sigma}) e_{\Gamma_{\sigma^{-1}}} \mathcal L(\Gamma)=\mathcal L(\Gamma)e_{\Gamma_{\sigma}}(W_{\sigma})  \mathcal L(\Gamma),\quad \sigma \in G.  $$

We identify the partial isometry $W_{\sigma}$  with the partial isometry
$$(\sigma \otimes \sigma^{-1})e_{\Gamma_{\sigma^{-1}}} = e_{\Gamma_{\sigma}}(\sigma\otimes \sigma^{-1}),\quad \sigma \in G.$$
In this identification,  the algebra $\mathcal B_{\infty}$ is
generated by $\L(\Gamma)\otimes 1$, the partial isometries $(W_{\sigma})_{\sigma \in G} $ and the Jones projections $(e_{\Gamma_\sigma})_{\sigma \in G}$.

Then $\B_\infty$ is isomorphic to a corner of the reduced von Neumann algebra crossed product
$\mathcal L((G \times G^{\rm op}) \rtimes L^{\infty}(\G, \mu))$, as follows:
\begin{equation}\label{isob}
\B_{\infty}= \chi_{\overline{\Gamma}}(\mathcal L((G \times G^{\rm op}) \rtimes L^{\infty}(\G, \mu)))\chi_{\overline{\Gamma}}.
 \end{equation}

 %
We recall that by Remark \ref{groupoid} it follows that
  using the groupoid  action of $G \times G^{\rm op}$ on $K$, we obtain the isomorphism
  \begin{equation}\label{isob1}
\B_\infty\cong \mathcal L((G \times G^{\rm op}) \rtimes L^\infty (K, \mu)),
   \end{equation}
    with unit identified to $\chi_{\overline{\Gamma}}$.

For $\sigma\in G$, using the definition of the isomorphism $\theta_\sigma$, we have that
 $$\pi(\sigma) \in{\Int}_{\theta_\sigma}(\pi(\Gamma_{\sigma^{-1}}), \pi(\Gamma_{\sigma})).$$
Let
$$s_i \in \mathcal L(\Gamma) \cong \pi(\Gamma)'', \quad  i=1,\dots,[\Gamma:\Gamma_\sigma],$$
be a Pimsner-Popa basis for the subfactor  inclusion  $\mathcal L(\Gamma_{\sigma}) \subseteq \mathcal L(\Gamma)$.
For simplicity, we let $(s_i)$  be a    system of  $\Gamma_{\sigma}$-left coset representatives. We let
$$t_i = \pi(\sigma)s_i,\quad  i=1,\dots,[\Gamma:\Gamma_\sigma].$$

With the  identification for  the space of intertwiners used in formula (\ref{pp}), each intertwining operator $\pi(\sigma)$, corresponds, because of  Lemma \ref{oldbg3},  to
 \begin{equation}\label{cosets}
\sum_i t_i^{\ast}W_{\theta_\sigma}s_i\in \B_\infty.
\end {equation}

Consequently, using the isomorphism   from formula (\ref{isob1}), the unitaries  $\pi(\sigma)$ correspond to the  the unitary elements:
\begin{equation}\label{unitary}
u(\sigma) = \mathop{\sum}\limits_{i} (t_i^{\ast}\otimes 1)(\sigma \otimes \sigma^{-1})e_{\Gamma_{\sigma^{-1}}}(s_i \otimes 1),\quad \sigma \in G.
\end{equation}
The  above formula gives, for every $\sigma\in G$,
$$
u(\sigma)=\mathop{\sum}\limits_{i} \chi_{\overline{\Gamma}}(t_i^{\ast} \otimes 1)(\sigma \otimes \sigma^{-1})(s_i \otimes 1) \chi_{\overline{\Gamma}}
= \mathop{\sum}\limits_{i} \chi_{\overline{\Gamma}} (t_i^{\ast} \sigma s_i) \otimes \sigma^{-1}\chi_{\overline{\Gamma}}.
$$

We denote by  $t^{\Gamma\sigma\Gamma}\in \mathcal L(G)$ the sum
$\mathop{\sum}\limits_{i} t_i^{\ast} \sigma s_i$. Note that  this is exactly the expression used in formula \eqref{t}.
 Consequently, we have the following expression for the unitary $u(\sigma)$:
\begin{equation}\label{u1}
u(\sigma) = \chi_{\overline{\Gamma}} (t^{\Gamma\sigma\Gamma} \otimes 1)(1 \otimes \sigma^{-1})\chi_{\overline{\Gamma}}, \quad \sigma \in G.
\end{equation}

Recall that  $\sigma\rightarrow \pi(\sigma)$ is a unitary (eventually projective) representation of $G$. Hence, by  using the product operation introduced in  formula (\ref{compbimodules})
from Lemma \ref{product}, it follows that formula (\ref{u1}) defines  a unitary representation $u$ of $G^{\rm op}$ into
$$\mathcal L((G \times G) \rtimes L^\infty (K, \mu)) = \chi_{\overline{\Gamma}}\mathcal L((G \times G^{\rm op}) \rtimes L^{\infty}(\G, \mu))\chi_{\overline{\Gamma}}.$$

We note that, a priori,  in the above construction we are initially working with a   larger C$^\ast$-norm  on the crossed
product $C^\ast$-algebra, coming from the full crossed product C$^\ast$-algebra $$C^\ast((G \times G) \rtimes L^\infty (K, \mu)).$$

The initial choice of this  $C^\ast$-algebra  crossed product norm corresponds  to the fact that we are working initially with a representation of the above groupoid crossed product $C^\ast$-algebra
into $B(\ell^2(\Gamma))$. This representation is obtained  by letting $G \times G^{\rm op}$ act by left and right multiplication operators on $\ell^2(\Gamma)$.
This  is due to the fact  that the values of the unitary $u(\sigma)$ belong in the first identification to bimodules of the form
$$L^2(\mathcal L(\Gamma)(W_{\sigma}) e_{\Gamma_{\sigma^{-1}}} \mathcal L(\Gamma)).$$
\noindent On this bimodules, we have a left and a right action of $G$ by multiplication.

On the other hand, the support of $u(\sigma)$ on the component corresponding to $G^{\rm op}$ is a singleton, so with respect to the component of $G^{\rm op}$ we may use the reduced C$^\ast$-norm
topology. Similarly, corresponding to the $G$-component  of the crossed product, the values of the unitary $u(\sigma)$ are already in $l^2(\Gamma\sigma\Gamma$).

Hence  the unitary representation $u$ of $G$ takes values in the reduced von Neumann algebra crossed product. Consequently, we may use indeed for the  algebra $\B_\infty$
the reduced groupoid $C^\ast$-algebra norm. Thus the representation $u$ has values in the  von Neumann  reduced  groupoid crossed product (\cite{Re}).
%
  \end{proof}

\section {The classification of the representations $\pi$ up to unitary conjugacy. Proof of Theorem \ref{liftout} }\label{unicitatea}

In this section we prove that the unitary representations $\pi$ of the group $G$ with the property that the restriction of $\pi$ to $\Gamma$ is unitary equivalent to  the left regular representation ($D_\pi=1$) are classified,
up to unitary conjugacy,  by the first group of cohomology of the group $\Gamma$ with values in the unitary group of the factor $P$.


\begin{rem}\label{unicx}

Let  $\pi$ and $ \pi_1$ be two unitary representations of $G$  such that $\pi|_\Gamma= \pi_1|_\Gamma= \lambda_\Gamma$. Assume there exists a unitary $\tilde U$ conjugating $\pi$ with $\pi_1$.
Then necessary $\tilde U\in \mathcal R(\Gamma)$. Using the canonical anti-isomorphism $\mathcal R(\Gamma)\cong \mathcal L(\Gamma)$, the unitary $\tilde U$ corresponds to    a unitary operator $U\in  \mathcal L(\Gamma)$, with the properties introduced bellow, in formula (\ref{tunic}).
Let $t$, respectively $t_1$, be the representations of the Hecke algebra constructed in formula (\ref{t}) that are respectively associated with the representations $\pi$, respectively $\pi_1$.

Then:
\begin{equation}\label{tunic}
t_1(\Gamma\sigma\Gamma)=U[ t(\Gamma\sigma\Gamma)]U^\ast,\quad \sigma \in G.
\end{equation}

This identity  corresponds to the fact that in order to compute the coefficients for the representation $t_1$ as in formula
(\ref{t}), one has to substitute the cyclic and separating trace vector $1$, used to compute the coefficients of $t$, by the vector $U\in \mathcal L(\Gamma)\subseteq   \ell^2(\Gamma)\cong H$.

The converse also holds true: Let  $\pi,\pi_1$ be two unitary representations of $G$,  whose restriction to $\Gamma$ is the left regular representation.
Let $t,t_1$ be the associated representations of the Hecke algebra introduced in formula (\ref{t}). Assume that the equation (\ref{tunic}) holds true for $U$ a unitary operator
in $\mathcal L(\Gamma)$.  Using the canonical anti-isomorphism $\mathcal L(\Gamma) \cong \mathcal R(\Gamma)$ we obtain a unitary
$\tilde U\in \mathcal R(\Gamma)=\pi(\Gamma)' = \pi_1(\Gamma)'$.  Then the unitary $\tilde U $ conjugates  the unitary  representation $\pi$ and $\pi_1$ of $G$.

 Indeed the equality in formula (\ref{tunic}) implies that the matrix coefficients of the two representations coincide.

\end{rem}

We prove in this section that the homomorphism $\alpha :G\rightarrow \operatorname{Aut}(P)$, constructed in Theorem \ref{split}, is uniquely determined up to cocycle conjugacy.
Moreover, the representations $\pi$ are classified, up to unitary conjugacy, by the cohomology group $H_\alpha^1(\Gamma, \mathcal U(P))$.

Recall that $u:G\rightarrow \mathcal U(P)$ is a 1-cocycle with respect to the action $\alpha$ if
$$u_{g_1g_2} = u_{g_1}\alpha_{g_1}(u_{g_2}),\quad g_1, g_2 \in G.$$

\noindent We have:

\vspace{0.5cm}
\begin{prop}\label{unicity}
Let $\pi$ be a unitary (projective) representation of  the group $G$ as above. Let $\alpha :G \rightarrow \operatorname{Aut} (P)$ be the homomorphism
constructed in Corollary  \ref{defalpha}.

\noindent Then:
\item{(i)}
 Any other homomorphism $\tilde{\alpha}:G\rightarrow \operatorname{Aut}(P)$ obtained from a similar splitting data as in Theorem \ref{split}, is of the form
 $\tilde{\alpha}_g = {\rm Ad}(u_g) \alpha_g$, where $u_g \in \U(P)$ is a 1-cocycle of $G$ with respect to $\alpha_g$, with values in the unitary group $\U(P)$.

 \item{(ii)}
The unitary representations $\pi'$ of $G$, whose restrictions to $\Gamma$ are unitarily equivalent to the left regular representation of $\Gamma$ are classified,
up to the unitary conjugacy, by the first cohomology group $H^1_{\alpha}(G, \mathcal U(P)).$

\item{(iii)} A cocycle $c$ in the cohomology group $H^1_{\alpha}(G, \mathcal U(P))$  is trivial if and only if $c$ is the coboundary of a unitary  element in
$\mathcal L(\Gamma)\subseteq P= \mathcal L(\Gamma\rtimes L^\infty (K,\mu))$.

 \end{prop}

\vspace{0.3cm}

\begin{proof} Let $\alpha=(\alpha_g)_{g \in G}$ be as in Theorem \ref{split}.
Then  $\mathcal{M} =  \mathcal L(G \rtimes L^{\infty}(\G, \mu))$ has the $G^{\rm op}$-equivariant tensor decomposition $P \otimes B(l^2(\Gamma \setminus G))$.
The canonical action of $G^{\rm op}$ onto $\mathcal{M}$ is
$$\alpha_g \otimes {\rm Ad}\rho_{\Gamma / G}(g),\quad  g \in G.$$

Elements in $\mathcal{M}$ are consequently identified with infinite matrices
 $$(p_{\Gamma\sigma_1, \Gamma\sigma_2})_{\Gamma\sigma_1, \Gamma\sigma_2 \in \Gamma \setminus G},$$ where the entries $p_{\Gamma\sigma_1, \Gamma\sigma_2}$ belong to  the algebra $P$.
Any other $G^{\rm op}$-equivariant matrix unit will be of the form $$(u(\Gamma\sigma_1)u(\Gamma\sigma_2)^\ast)_{\Gamma\sigma_1, \Gamma\sigma_2\in \Gamma \setminus G},$$
where $u(\Gamma\sigma)$ are unitaries in $P$ whenever $\Gamma\sigma \in \Gamma \setminus G$.

For every $\sigma \in G$, we identify the coset  $\Gamma\sigma$ with the characteristic function $\chi_{K\sigma}$.
Then the diagonal algebra
$$D = l^{\infty}(\Gamma \setminus G) \subseteq B(l^2(\Gamma \setminus G))\subseteq B(L^2(G \rtimes L^{\infty}(\G, \mu)) $$

 \noindent  is independent of the choice of an equivariant matrix unit.

 Thus  $(u(\Gamma\sigma))_{\Gamma\sigma\in \Gamma\setminus G}$ is  a unitary of the form
 $$w = \sum_{\Gamma\sigma\in \Gamma\setminus G} u(\Gamma\sigma)\otimes \chi_{\overline{\Gamma\sigma}}.$$
 Then the unitary $w$  belongs to the unitary group of the algebra
 $$D' \cap \mathcal L(G \rtimes L^{\infty}(\G, \mu)).$$
 We denote the initial matrix unit,  associated to the homomorphism $\alpha$, splitting the action of $G^{\rm op}$, by
 $$(v_{\Gamma\sigma_1, \Gamma\sigma_2})_{\Gamma\sigma_1,\Gamma\sigma_2\in \Gamma\backslash G}.$$
 We impose the $G^{\rm op}$-equivariance condition on the perturbed matrix unit
 $$\tilde{v}_{\Gamma\sigma_1, \Gamma\sigma_2} = u(\Gamma\sigma_1)^{\ast}u(\Gamma\sigma_2) \otimes v_{\Gamma\sigma_1, \Gamma\sigma_2}, \quad {\Gamma\sigma_1,\Gamma\sigma_2\in \Gamma\backslash G}.$$

We consider the type I$_\infty$ von Neumann algebra generated by the matrix units introduced above:
$$B = \{ v_{\Gamma\sigma_1, \Gamma\sigma_2}| \ {\Gamma\sigma_1,\Gamma\sigma_2\in \Gamma\backslash G} \}'',\quad
 \tilde{B} = \{ \tilde {v}_{\Gamma\sigma_1 \Gamma\sigma_2}|\ {\Gamma\sigma_1,\Gamma\sigma_2\in \Gamma\backslash G} \}''.$$
\noindent  For every double coset  $\Gamma\sigma\Gamma$ in G,  let
\begin{equation}\label{xsigma}
X^{\Gamma\sigma\Gamma} = \mathop{\sum}\limits_{\Gamma\sigma_1\sigma_2^{-1}\Gamma = \Gamma\sigma\Gamma} v_{\Gamma\sigma_1, \Gamma\sigma_2},
\quad
 \tilde{X}^{\Gamma\sigma\Gamma} = \mathop{\sum}\limits_{\Gamma\sigma_1\sigma_2^{-1}\Gamma = \Gamma\sigma\Gamma} \tilde{v}_{\Gamma\sigma_1, \sigma_2\Gamma}.
 \end{equation}
These are  the elements constructed in formula (\ref{defx}) in the proof of the implication (iv) $\Rightarrow$ (ii) in Theorem \ref{split}.

Consider  the correspondence mapping $[\Gamma\sigma\Gamma] \to X^{\Gamma\sigma\Gamma}$ and the correspondence mapping $[\Gamma\sigma\Gamma] \to \tilde{X}^{\Gamma\sigma\Gamma}$, where  $[\Gamma\sigma\Gamma]$
runs over double cosets. Then both correspondences  extend by linearity  to $\ast$-representations of the algebra $\H_0=\C(\Gamma \setminus G / \Gamma)$. As explained in the proof of Theorem \ref{split}, this representations
extend to representations of the reduced Hecke $C^{\ast}$-algebra $\H$, with values into the algebra $\mathcal L(G)$.
Then we  have:
\begin{equation}\label{vtov}
 w \big[(v_{\Gamma\sigma_1, \Gamma\sigma_2})_{\Gamma\sigma_1,\Gamma\sigma_2\in \Gamma\backslash G}\big ]w^\ast= (\tilde {v}_{\Gamma\sigma_1, \Gamma\sigma_2})_{\Gamma\sigma_1,\Gamma\sigma_2\in \Gamma\backslash G}.
 \end{equation}
Consequently
$$wBw^{\ast} = \tilde{B},$$
\noindent and
 $$wX^{\Gamma\sigma\Gamma}w^{\ast} = \tilde{X}^{\Gamma\sigma\Gamma},$$ for all double cosets $[\Gamma\sigma\Gamma]$.

The $G^{\rm op}$-invariance of the perturbed  matrix unit $(\tilde {v}_{\Gamma\sigma_1, \Gamma\sigma_2})_{\Gamma\sigma_1,\Gamma\sigma_2\in \Gamma\backslash G}$ implies that the unitary
$\beta_g(w)$ has the same properties as $w$. Consequently $\beta_g(w)^{\ast}w$ belongs to $B' = P \otimes I$.

It follows that
$$c(g) = \beta_g(w)^{\ast}w,\quad g\in G,$$
is a 1-cocycle for the group $G$ with respect to $\alpha_g$, with values in $\mathcal U(P)$.

Let  $\tilde{\alpha}: G\rightarrow \operatorname{Aut} (P)$ be the homomorphism associated in Theorem \ref{split} to the $\gop$-matrix unit
$(\tilde{v}_{\Gamma\sigma_1, {\Gamma\sigma_2}}) _{\Gamma\sigma_1,\Gamma\sigma_2\in \Gamma\backslash G}.$  Then
$$\tilde{\alpha}_g = {\rm Ad }c(g)\alpha_g, \quad  g \in G.$$

\noindent This completes the proof of part (i) in the statement.

\vspace{0.3cm}

If, in  the first cohomology group $H_{\alpha}^1(G, P)$, the  1-cocycle $c$ with values in $P \cong P \otimes 1 = B' \subseteq D'$ vanishes,
 then it follows that there exists a unitary $p\in P$ such that
 \begin{equation}\label{trivialitate}
  c(g) = \beta_g(p^{\ast})p, \quad g\in G.
  \end{equation}

Thus
$$\beta_g(w^{\ast})w = \beta_g(p^{\ast})p,$$
for $g$ in $G$. Hence $$\beta_g(wp^{\ast}) = wp^{\ast},\ g\in G.$$
\noindent Consequently the unitary operator $wp^{\ast}$ belongs to the subspace of   $\gop$-invariant
elements of $\mathcal L(G\rtimes L^\infty(\G,\mu))$. Hence
$$ wp^{\ast}\in \mathcal L(G).$$
Since both $w, p$ belong to $D'$ it follows that
$$wp^{\ast} \in \mathcal L(G) \cap l^{\infty}(\Gamma \setminus G)'=\mathcal L(\Gamma).$$
\noindent  We obtain  that there exists a   unitary operator  $x\in\mathcal L(\Gamma)$ such that $w = xp$. Recall that,
by definition, the unitary
$p\in P$ commutes with the matrix unit $(v_{\Gamma\sigma_1, \Gamma\sigma_2})_{\Gamma\sigma_1,\Gamma\sigma_2\in\Gamma\backslash G}$.
 Using equation (\ref{vtov}) we obtain that
 \begin{equation*}
 \begin{split}
 w \big[(v_{\Gamma\sigma_1, \Gamma\sigma_2})_{\Gamma\sigma_1,\Gamma\sigma_2\in \Gamma\backslash G}\big ]w^\ast & =
 x \big[(v_{\Gamma\sigma_1, \Gamma\sigma_2})_{\Gamma\sigma_1,\Gamma\sigma_2\in \Gamma\backslash G}\big ]x^\ast \\ & =
 (\tilde {v}_{\Gamma\sigma_1, \Gamma\sigma_2})_{\Gamma\sigma_1,\Gamma\sigma_2\in \Gamma\backslash G},
 \end{split}
 \end{equation*}
\noindent
and
$$wBw^{\ast} = x{B}x^{\ast}=\tilde B.$$
 Using equations (\ref{xsigma}) we obtain that
$$\tilde{X}^{\Gamma\sigma\Gamma} = x X^{\Gamma\sigma\Gamma}x^{\ast}, \quad \sigma\in G.$$

Recall that
 in Theorem \ref{split} we proved that  the elements constructed as in the formulae (\ref{xsigma}) are the same as the elements
 $t(\Gamma\sigma\Gamma)$, $t'(\Gamma\sigma\Gamma)$, $\sigma\in G$, associated to the unitary representations $\pi$ and $\pi'$  in part (ii) of the above mentioned theorem (formula \eqref{t}).

Using Remark \ref{unicx}, it follows that the unitary representations $\pi$ and $\pi'$ are conjugated by a unitary. The last statement is done in two steps: first we may assume,
up to conjugation by a unitary operator, that the unitary representations $\pi$ and $\pi'$, when restricted to $\Gamma$, are equal  the left regular representation of the group $\Gamma$.
In a second step, we use the unitary $x$ constructed above to verify the hypothesis of the converse statement in the above mentioned remark.
This completes the proof of part (ii) in the statement.

To prove part (iii) we use the above arguments.
Assume that $c$ is the cocycle relating two homomorphisms $\alpha$ and $\tilde \alpha$  of $G$ as above.
The cocycle $c(g)$ relating the homomorphisms $\alpha$ and
$\tilde \alpha$ is trivial if and only if there exists a unitary $p\in P$ such that equation (\ref{trivialitate}) holds true.

From the previous part it follows that
   $\tilde{\alpha}_g$ differs from $\alpha_g$ by conjugation with   Ad($x$),  where $x$ is a  unitary element in $\mathcal L(\Gamma)$.

   The homomorphisms $\alpha$ and $\tilde\alpha$ of $G$ into Aut($\mathcal L(\Gamma\rtimes L^\infty(K,\mu)))$  are  associated to the representations $\pi$ and respectively $\pi'$ by the correspondence introduced in Corollary \ref{defalpha}. The equivalent construction for the homomorphisms $\alpha$ and $\tilde \alpha$ described in the introduction, is realized by conjugating first the representations $\pi$, $\pi'$ by a unitary so that $\pi|_\Gamma$, $\pi'|_\Gamma$ are the left regular representation of $\Gamma$.

We have therefore proved that if $c$ is trivial, then the unitary $p$ in  the above mentioned formula may be replaced by a unitary $x$ in $\mathcal L(\Gamma)\subseteq P.$
Thus
$$c(g)= \beta_g(x^{\ast})x, \quad g\in G.$$
 \end{proof}

In the next proposition we prove that the value at $\sigma \in G$ of the 1-cocycles relating two unitary representations as in part (iii) of Theorem \ref{split},
have an expression that depends only on the double coset $\Gamma\sigma\Gamma$.

 \begin{prop}\label{cociclicitate}
  Given a representation $\pi$ as in Theorem \ref{split} (i), consider as in  part (iii) of the above mentioned theorem, the associated unitary representation $u:G\rightarrow \mathcal U(\mathcal B_\infty)$.
   Let $u'$ be the unitary representation of $G$ into $\mathcal B_\infty$ corresponding  to another  representation $\pi'$ of $G$, as above.
Let $v:G\rightarrow \mathcal U(\A_\infty)$ in $Z^1_\alpha (G, \mathcal U(\mathcal A_\infty))$  be the 1-cocycle
  defined by
 \begin{equation}\label{uprime}
u'(\sigma)= v(\sigma)u(\sigma),\quad \sigma \in G.
 \end{equation}

 \noindent Then:
 \item{(i)}
The 1-cocycle  $v$ has  the following form:
For each double coset $\Gamma\sigma\Gamma$,  there exist Pimsner-Popa bases $(x_i^{[\Gamma\sigma\Gamma]})$, $(y_i^{[\Gamma\sigma\Gamma]})$ in $\mathcal L(\Gamma)$, $i=1,2,...,[\Gamma:\Gamma_\sigma]$,
for the inclusions $\mathcal L(\Gamma_\sigma)\subseteq \mathcal L(\Gamma)$ and respectively   $\mathcal L(\Gamma_{\sigma^{-1}})\subseteq \mathcal L(\Gamma)$,
 such that
 \begin{equation}\label{vvv}
v(\sigma)= \sum_{i=1}^{[\Gamma:\Gamma_\sigma]}x_i^{[\Gamma\sigma\Gamma]}\chi_{{K_\sigma}}y_i^{[\Gamma\sigma\Gamma]}\otimes 1\in \mathcal A_\infty \subseteq \mathcal B_\infty.
\end{equation}

\item{(ii)}
The unitary representation  $\pi$ is unitarily equivalent to $\pi'$  if and only if $v$
is a coboundary. In this case, there exist a unitary $w\in\mathcal L(\Gamma)$ such that
$$v(g)= w^\ast [u(g)(w)u(g)^\ast], \quad  g \in G.$$
\end{prop}
\begin{proof}

We may assume that the unitary representations $\pi,\pi'$ act on the same Hilbert space and that $\pi|_\Gamma=\pi'|_\Gamma$.
The formula (\ref{uprime}) is a consequence of the previous statement and of Proposition \ref{doublecross}, since by the ergodicity assumptions,
the relative commutant of $\mathcal A_\infty$ in $\mathcal B_\infty$ is trivial.

For $\sigma\in G$, let
$$V(\sigma)=\pi'(\sigma)\pi(\sigma)^\ast.$$
By definition, we have that
 \begin{equation}\label{echivalenta}
V(\sigma)\in \pi(\Gamma_\sigma)'.
\end{equation}
It follows that
 $V(\sigma)$ depends only on the coset $\sigma\Gamma$ for $\sigma \in G$.
 We have
 \begin{equation}\label{sym}
V(\gamma\sigma)=\pi(\gamma) V(\sigma)\pi (\gamma)^\ast,\quad \gamma\in\Gamma, \sigma\in G.
\end{equation}
Using the identification in Section \ref{algebraization} of intertwiners with elements in the algebra $\mathcal B_\infty$ it follows that $V(\sigma)$,
as a self-intertwiner of the von Neumann algebra $\pi(\Gamma_\sigma)'$ corresponds to
a 1-cocycle $v$ as in formula (\ref{uprime}).

 Using the identification in formula (\ref{asigma}), formula (\ref{sym}) translates into a similar formula corresponding  to the canonical right action of
 $\gamma\in \Gamma$, mapping $\mathcal A_{\Gamma\sigma}$ into $\mathcal A_{\Gamma\sigma\gamma}$.
 This fact and equation (\ref{sym}) are translated, using the terminology from the previous section, into the content of statement (i).

 The statement (ii) is a direct consequence of Proposition \ref{unicity}.
\end{proof}

\begin{proof}[Proof of Theorem \ref{liftout}]
The proof of the implication (i) $\Rightarrow$ (iii) in Theorem \ref{split} proves that the homomorphism $\alpha$ constructed in Corollary \ref{defalpha}
coincides with the homomorphism $\alpha: G\rightarrow \operatorname{Aut}(P)$ introduced in formula (\ref{alpha11}).
The latter is automatically a lifting of the map $\Phi$ to a homomorphism from $G$ into Aut$(P)$. This completes the proof of part (i).

Part (ii) is a direct consequence of formula (\ref{echivalenta}).
This, combined with the result in Proposition \ref{unicity}, proves part (iii).
\end{proof}

\vspace{0.5cm}
We exemplify in the next statement the set of conditions required to construct a unitary representation $\pi$, in the case $G = PGL_2(\Z[\frac{1}{p}])$.

For  $n\in \N$, we define
$$\sigma_{p^n}=\left(
\begin{matrix}
p^n&0\\0&1
\end{matrix}\right).
$$

\begin{prop}
In the case $G = PGL_2(\Z[\frac{1}{p}])$, $\Gamma = PSL_2(\Z)$, $p$ a prime number, the space of cosets  of $\Gamma$ in $G$ has a $p+1$ homogeneous  tree structure.
Then, to get a $G$-invariant matrix unit in $\mathcal L(G \rtimes L^{\infty}(\G, \mu))$, as in the Theorem \ref{split}, it is sufficient to find $X = X^{\ast}\in\mathcal L(G) \cap l^2(\Gamma\sigma_p\Gamma)$,
such that the following two properties hold true:
\item{(i)}
 $\chi_{\overline{\Gamma}}X\chi_{\overline{\Gamma\sigma_p}}$ is an isometry from $\chi_{\overline{\Gamma}}$ onto $\chi_{\overline{\Gamma\sigma_p}}$.

\item {(ii)} Let $(s_i)\subseteq \Gamma$ be a system of left coset representatives for $\Gamma_{\sigma_p}$ in $\Gamma$. We require that
$\chi_{\Gamma} X \chi_{\overline{\Gamma\sigma_p}}$ is the adjoint of
$$\chi_{\overline{\Gamma\sigma_p}}\beta_{(\sigma_ps_i)}(X)\chi_{\Gamma},$$
if $\Gamma\sigma_p^{-1} = \Gamma\sigma_p s_i$.

\end{prop}

 \begin {proof}
 The second requirement corresponds to the requirement that the representation defined by $X$ be unitary.

  We let $v_{\Gamma\sigma_1, \Gamma\sigma_2} = \beta_{\sigma_1}(\chi_{\overline{\Gamma}}X\chi_{\overline{\Gamma\sigma_p}})$ if $\Gamma \sigma_1\sigma_2^{-1}\Gamma = [\Gamma\sigma_p\Gamma]$,
  and the tree structure implies that we may define a $G$ - equivariant matrix unit by defining, for $\Gamma \sigma_1\sigma_n^{-1}\Gamma = [\Gamma\sigma_{p^n}\Gamma]$,
$$
v_{\Gamma\sigma_0, \Gamma\sigma_n} = \mathop\prod\limits_{i} v_{\Gamma\sigma_i\Gamma\sigma_{i+1}},
$$

\noindent where $\Gamma\sigma_i\Gamma\sigma_{i+1}^{-1}\Gamma = [\Gamma\sigma_p\Gamma]$, and $(\Gamma\sigma_i)$ are the edges in the path on the tree connecting $\Gamma\sigma_0$ to $\Gamma\sigma_n$.

Condition (ii) implies that this is well defined.
\end{proof}

\bigskip

\section{Generalizing to the situation $D_\pi \neq 1$}\label{appendix}

In this section we construct the  technical tools necessary to generalize the previous results to
the case of unitary representation $\pi$ that restricted to $\Gamma$ is a multiple (not necessary an integer) of the left regular representation of $\Gamma$.

Recall that the Schlichting completion $\G$ of the group $G$ is a locally compact group, that is totally disconnected and contains $G$ as a dense subgroup.
The Haar measure on $\G$ is denoted as before by $\mu$.
We denote by $\S$ the lattice of finite index  subgroups of $K$ generated  by subgroups of the form $K\cap\sigma K\sigma^{-1}, \sigma\in G$.

We consider the following reduced crossed product $C^\ast$ algebras:
\begin{equation}\label{A}
\mathcal A=C^\ast_{\rm red}(\G \rtimes L^\infty(\G,\mu)),
\end{equation}
\begin{equation} \label{B}
\mathcal B=C^\ast_{\rm red}((\G\times G^{\rm op})\rtimes L^\infty(\G,\mu)).
\end{equation}
For a measurable subset $A$ of $\G$, we identify the characteristic function with the corresponding projection   $\chi_A\in  L^\infty(\G,\mu)$.

By $L(\chi_A)$ we denote the corresponding left convolutor in $C^\ast_{\rm red}(\G)$. By $R(\chi_A)$ we denote the corresponding right convolutor,
viewed as an element in $C^\ast_{\rm red}(\G^{\rm op})$.

We also consider the $C^\ast$-algebras:
\begin {equation}\label{C}
\mathcal C=\chi_K {\mathcal B}\chi_K\cong C^\ast_{\rm red}((\G\times G^{\rm op})\rtimes L^\infty(K,\mu))\quad \mbox{\rm and}
\end{equation}
\begin{equation}\label{D}
\begin{split}
\mathcal D & =\chi_K C^\ast_{\rm red}((\G\times \G^{\rm op})\rtimes L^\infty(\G,\mu))\chi_K \\   &  \cong
C^\ast_{\rm red}((\G\times \G^{\rm op})\rtimes L^\infty(K,\mu)).
\end{split}
\end{equation}
In the above equations, the groups  $\G\times G^{\rm op}$ and $\G\times \G^{\rm op}$ have an obvious action by partial measure preserving isomorphisms on $K$, and the second algebra in both equations
is a reduced crossed product groupoid $C^\ast$-algebra. Let $\mathcal A_0$, $\mathcal B_0$,  $\mathcal C_0$, $\D_0$ be  the corresponding involutive unital subalgebras of the C$^\ast$-algebras $\A$,
$\B$, $\mathcal C$, $\D$ that are generated by the characteristic functions of cosets in $\G$ of subgroups in $\S$, by convolutors with such characteristic functions and, in the case of the algebra
$\mathcal B_0$,  by the elements of the group algebra, over $\C$, of the group $G$.

\begin{rem} The multiplication rule in  the C$^\ast$-crossed product algebra  $\mathcal A=C^\ast_{\rm red}(\G\rtimes L^\infty(\G,\mu))$ is determined as follows:
\item{(i)} For every $\Gamma_0\in \S$, let $\chi_{K_0}$ be the characteristic function of a subgroup $\K_0=\overline {\Gamma_0}$. Then for all $g\in G$ we have
$$\L(\chi_{K_0}) \chi_{K_0g}= \chi_{K_0g} \L(\chi_{K_0}).$$

\item {(ii)} If $A, B$ are measurable subsets of $\G$, such that
there exists a subgroup $K_0$ as above and finite families $(s_i)$, $(t_j)$ in $G$ such that
$$A=\bigcup_{i} s_iK_0,\quad B=\bigcup_j K_0 t_j,$$
\noindent then
\begin{equation}\label{rulem}
L(\chi_A)\chi_B=\sum_{i,j} L(\chi_{s_iK_0})\chi_ {K_0 t_j}= \sum_{i,j} \chi_ {s_i^{-1}K_0 t_j} L(\chi_{s_iK_0}).
\end{equation}
\end{rem}
\begin {proof}
Property (i) is obvious since $K_0$ is a subgroup. Property (ii) is a direct consequence of Property (i).
\end{proof}

First we prove that the unitary representation $u$ and the representation $\Psi$ from Theorems \ref{split} and \ref{heckes}  have an abstract  counterpart, taking values in the algebras
$\mathcal C$ and $\mathcal D$.

\begin{thm}\label{absrep}
For $\sigma \in G$, let $\chi_{K\sigma K}$ be the characteristic function of the double coset
$K\sigma K$. Define
\begin{equation}\label{UU}
U(\sigma)=\chi_K(L(\chi_{K\sigma K})\otimes \sigma^{-1})\chi_K\in \mathcal C,
\end{equation}
\begin{equation}\label{psi0}
\Psi_0(\Gamma\sigma\Gamma)=\chi_K(L(\chi_{K\sigma K})\otimes R(\chi_{K\sigma^{-1} K}))\chi_K \in \D.
\end{equation}

\noindent Then:

\item {(i) } $U$ is a unitary representation of $G$ into $\U(\mathcal C)$.

\item {(ii) } The correspondence associating to a double coset $\Gamma\sigma\Gamma$ of $\Gamma$ in $G$ the element $\Psi_0(\Gamma\sigma\Gamma)$ extends by linearity to a
$\ast$-representation of the Hecke algebra $\mathcal H_0$ into $\D$.
\end{thm}

\begin{proof}
Fix an element $\sigma_1\in G$. Let $(s_i)$ be a family of coset representatives for
$\Gamma_{\sigma_1}= \sigma_1\Gamma\sigma_1^{-1}\cap \Gamma$ in $\Gamma$. For $\theta \in G$, let $K_\theta$ be the subgroup $\theta K\theta^{-1}\cap K$.
It is then obvious that
\begin{equation}\label{uiformula}
\begin{split}
U(\sigma_1) & =\sum_i(L(\chi_{K\sigma_1 s_i})\otimes \sigma_1^{-1})\chi_{K\cap s_i^{-1}\sigma_1^{-1}K\sigma_1} \\   &  =
\sum_i(L(\chi_{K\sigma_1 s_i})\otimes \sigma_1^{-1})
\chi_{ s_i^{-1}K_{\sigma_1^{-1}}}.
\end{split}
\end{equation}
Let $\sigma_2$ be another element in $G$. Let $(r_j)$ be a system of coset representatives
for $\Gamma_{\sigma_2}$ in $\Gamma$.
Then
$$U(\sigma_2)U(\sigma_1)=\chi_K\big[\sum_{i,j}(L(\chi_{K\sigma_2r_j\sigma_1 s_i})\otimes \sigma_1^{-1}\sigma_2^{-1})
\chi_{K\cap s_i^{-1}\sigma_1^{-1}K\sigma_1}\big].
$$
This is further equal to
\begin{equation}\label{bij}
\sum_{i,j}(L(\chi_{K\sigma_2r_j\sigma_1 s_i})\otimes \sigma_1^{-1}\sigma_2^{-1})\chi_{B_{i,j}},
\end{equation}
\noindent where
$$B_{i,j}=\{k \in K | \sigma_1 s_i k\sigma_1^{-1}\in K, \sigma_2 r_j [\sigma_1 s_i k\sigma_1^{-1}]\sigma_2^{-1}\in K\}.$$
Thus, for a fixed $i$, $B_{i,j} \subseteq s_i^{-1}K_{\sigma_1^{-1}} $ is the preimage, through the map
$$\sigma_1 s_i\  \cdot \ \sigma_1^{-1},$$ of the partition of the subgroup $K_{\sigma_1}$  given by
$$(r_j^{-1}K_{\sigma_2}\cap  K_{\sigma_1})_{i,j}.$$
Since the cosets $s_i^{-1}K_{\sigma_1^{-1}}$ are disjoint, $(B_{i,j})_{i,j}$ is a partition of $K$.

On the other hand a term of the form
\begin{equation}\label{pv}
\chi_K(L(\chi_{K\alpha}\otimes \sigma_1^{-1}\sigma_2^{-1})\chi_K,
\end{equation}
is different from zero if and only if $K\alpha$ is contained in $K\sigma_1\sigma_2K$. Hence $$\alpha=\sigma_2\sigma_1v,$$ where $v$ is a coset representative  for $\Gamma_{\sigma_2\sigma_1}$ in $\Gamma$.
Thus a non-zero term  in formula (\ref{pv}) is necessarily of the form:
\begin{equation}\label{pvv}
L(\chi_{Kv}\otimes \sigma_1^{-1}\sigma_2^{-1})\chi_{v^{-1}K_{(\sigma_2\sigma_1)^{-1}}}.
\end{equation}

When $v$ runs over the set of coset representatives for $\Gamma_{\sigma_2\sigma_1}$ in $\Gamma$, the sum of the terms in formula (\ref{pvv}) is
$$U(\sigma_2\sigma_1)= \chi_K(L(\chi_{K\sigma_1\sigma_2K}\otimes \sigma_1^{-1}\sigma_2^{-1})\chi_K.$$

Since both families $(v^{-1}K_{(\sigma_2\sigma_1)^{-1}})_v$ and $(B_{i,j})_{i,j}$  are partitions of $K$,
it follows that sum of the terms in formula (\ref{bij}) is also  equal to $U(\sigma_2\sigma_1)$, and hence $U$ is a unitary representation of $G$.

The second formula is proven by a similar argument (see also \cite{Ra5}) noting that using the above notations, one has that
$$\chi_K(L(\chi_{K\sigma K})\otimes R(\chi_{K\sigma^{-1} K}))\chi_K=
\sum_{i,j} L(\chi_{K\sigma s_i})\otimes R(\chi_{s^{-1}_j\sigma^{-1} K})\chi_{K\cap s^{-1}_i \sigma^{-1} K\sigma s_j}.$$
\end{proof}

\begin{defn}\label{cansos} Consider the following operator system inside $C^\ast_{\rm red}(\G)$:
$$\mathcal E=\ Sp\ \{ L(\chi_{\sigma_1K}) L(\chi_{K\sigma_2})| \sigma_1, \sigma_2\}\subseteq
C^\ast_{\rm red}(\G).$$
Obviously the Hecke algebra of double cosets
$$\H_0=\ Sp \ \{ L(\chi_{K\sigma K})| \sigma \in G\},$$
is a subspace of $\mathcal E$.

Using the product in the algebra $C^\ast_{\rm red}(\G)$, it follows that every double coset
$\chi_{K\sigma K}$ acts on $\mathcal E$, extending by linearity the correspondence
\begin{equation}\label{sosmult}
L(\chi_{\sigma_1K}) L(\chi_{K\sigma_2})\rightarrow
L(\chi_{\sigma_1K})L(\chi_{K\sigma K})L(  \chi_{K\sigma_2})\in \mathcal E,\quad \sigma_1,\sigma_2\in G.
\end{equation}


\end{defn}

\begin{defn}\label{somap}
Let $\mathcal U$ be an abstract involutive algebra over $\C$. Assume that
$\mathcal U$ contains $L^\infty(\G,\mu)$ and an involutive subalgebra $\mathcal V$ such that
$$\mathcal U=\  Sp\ [\mathcal V \cdot L^\infty(\G,\mu)]= \  Sp\   [L^\infty(\G,\mu) \cdot \mathcal V].$$
We say that a linear map $\Phi$ defined on $\mathcal A_0$ with values into $\mathcal U$ is $\S\O$-system map if the following properties hold true:

\item{ (i)}  The restriction of $\Phi$ preserves the operation described in Definition \ref{cansos}. Thus we assume that  for all $\sigma,\sigma_1,\sigma_2\in G$ we have
\begin{equation}\label{sospreserving}
\Phi (L(\chi_{\sigma_1K})L(\chi_{K\sigma K})L(  \chi_{K\sigma_2}))=\Phi(L(\chi_{\sigma_1K})) \Phi(L(\chi_{K\sigma K})) \Phi(L(\chi_{\sigma_2K})).
\end{equation}

\item{(ii)} The map $\Phi$ is $\ast$-preserving, that is, for all $\sigma \in G$, we have
$$\Phi(\chi_{K\sigma})^\ast=\Phi(\chi_{\sigma K}).$$

\item{(iii)} The map $\Phi$ is preserving the support in $L^\infty(\G,\mu)$, that is:
\begin{equation}\label{support}
\Phi(f_1X f_2)=f_1 \Phi(X) f_2, \quad  f_1,f_2 \in L^\infty(\G,\mu), X\in \mathcal E.
\end{equation}

\end{defn}

We recall the context from the paper \cite{Ra2}. Let $\pi$ be a unitary  representation of $G$ into a Hilbert space $H$.
Assume that $\pi|_\Gamma$ admits a closed  subspace $L\subseteq H$ such that
$$\pi(\gamma)L\perp L,\quad  \gamma\in\Gamma\setminus \{e\},$$
\noindent  and
$$\overline {\ Sp\ \{\pi(\gamma)L| \gamma \in \Gamma\} }=H.$$
We will call such a subspace a $\Gamma$-wandering generating subspace of $H$. In particular $\pi|_\Gamma$ is a multiple of the left regular representation $\lambda_\Gamma$
with multiplicity equal to dim$_{\mathbb C}L$. Assuming the above conditions and assuming that dim$_{\mathbb C}L$ is a finite integer, the following proposition was proved in \cite{Ra2}.
Let $\rho_\Gamma:\Gamma\rightarrow \mathcal U(\ell^2(\Gamma))$ be the left regular representation of $\Gamma$.
Let $\rho_G$ be the right regular representation of $G$. Let $\mathcal R(G)$, respectively $\mathcal R(\Gamma)$, be the II$_1$ von Neumann algebras generated by the right regular representations
of $G$ and $\Gamma$ respectively.

\begin{prop}\label{veche}
 (see \cite{Ra2}, \cite{Ra5}). Let $P_L$ be the orthogonal projection from $H$ onto $L$. For $A$ a coset of a subgroup in $\S$ define
\begin{equation}\label{generalos}
\Phi_\pi(L(\chi_A))=\sum_{\theta\in A}\rho(\theta)\otimes P_L\pi(\theta)P_L
\end{equation}
Then $\Phi$ takes values in $\mathcal R(G)\otimes B(L)$ and $\Phi|_\mathcal E$ is a representation of the operator system $\mathcal E$,
that  verifies properties (i) and (ii) from Definition \ref{somap}.

\end{prop}

Using this, we can easily show that one can construct a representation into $\mathcal R(G)\otimes B(L)\cong \mathcal L(G)\otimes B(L) $ of the unitary representation and the
Hecke algebra representation constructed in Theorem \ref{absrep}. Using the notations from the previous proposition we can prove the following:

\begin{thm}\label{abstractsetting}
(i) For every $\sigma\in G$, the formula
$$U(\sigma)=\chi_K(\Phi_\pi(\chi_{\Gamma\sigma\Gamma})\otimes \sigma^{-1})\chi_K\in
\mathcal L((G\otimes G^{\rm op})\rtimes K)\otimes B(L)$$
defines a unitary representation $U$ of $G$ into $\mathcal L((G\otimes G^{\rm op})\rtimes K)\otimes B(L)$.

\noindent (ii) The mapping
$$T(\chi_{\Gamma\sigma\Gamma})=\chi_K(\Phi(\chi_{\Gamma\sigma\Gamma})\otimes \Phi(\chi_{\Gamma\sigma^{-1}\Gamma}) )\chi_K\in \mathcal L((G\otimes G^{\rm op})\rtimes K)\otimes B(L) $$
extends by linearity to a representation of the Hecke algebra $\mathcal H_0$ into
$\mathcal L((G\otimes G^{\rm op})\rtimes K)\otimes B(L)$.
\end{thm}

\begin{proof} We extend $\Phi$  to the algebra generated by $\mathcal E$ and $L^\infty(\G,\mu)$ in $\mathcal A_0$. It is obvious that this can be done by letting $\Phi$ be equal to the identity on
$L^\infty(\G,\mu)$. Because of the formula (\ref {generalos}) and using the results from  Proposition
\ref{veche}, it follows that the extension of $\Phi$ verifies the conditions from Definition \ref{somap}.
Since the verifications of the statement of Theorem \ref{absrep} involve only identities of the type considered in Definition \ref{somap}, it follows that the extension of $\Phi$ to the corresponding algebras
translates the content of Theorem \ref{absrep} into the statement of the present theorem.
\end{proof}

\begin{rem}
Having constructed the unitary representation $U$ as in  part (iii) of Theorem \ref{split},  one may repeat ad-litteram, in this more general setting,  the proof of the statements from the previous sections.

The case where $\pi|_\Gamma$ is a finite, but not necessary  an integer multiple of the left regular representation, may be treated as in \cite{Ra2}.
One assumes that the representation $\pi$ admits a $G$-invariant subspace $H_0$ such that in the commutant of $\pi(\Gamma)$, the projection $P_0$ onto $H_0$ has trace $t$, a positive real number.
Then one considers the unitary representation $\pi_0(g)=P_0\pi(g)P_0$, $g\in G$, which corresponds to the case where the multiplicity of the left regular representation
$\lambda _G$ in $\pi_0|_{\Gamma}$ is equal to $t$ (in the sense of the Murray-von Neumann dimension theory).
In this case all the above arguments may be repeated, replacing formula (\ref{generalos}) by the formula
\begin{equation}\label{mostgeneral}
\Phi_\pi(L(\chi_A))=\sum_{\theta\in A}\rho(\theta)\otimes P_L\pi_0(\theta)P_L.
\end{equation}
The case where the representation $\pi$ (and hence $\pi_0$) is unitary, projective is treated as in \cite{Ra1} (see also \cite{Ra3}).
\end{rem}

We note a few remarks regarding the relation of this paper with the results in \cite {Ra5}.

\begin{rem}\label{relation1}
Consider a general representation $\pi_0$ on a Hilbert space $H_0$ such that $\pi_0|_\Gamma$ is a multiple of the left regular representation $\lambda_\Gamma$.
In \cite{Ra5} we proved that there is a corresponding representation $\overline{\pi_0}^{\rm ad}$ on a Hilbert space $\overline {H_0}^{\rm ad}$,
which encodes all the information on the action
of the representation $\pi_0$ on the extended vector space of vectors that are invariant to subgroups in $\S$.

The representation $\overline{\pi_0}^{\rm ad}$ is obtained through an adelic completion procedure, and therefore we will use the symbol $\overline{\ \cdot\ }^{\rm ad}$
to denote an object obtained through such a construction. For the conjugate Hilbert spaces and representation, we will use as usual the bar symbol.
The essential property of the representation $\overline{\pi_0}^{\rm ad}$ is the fact  that it extends to a C$^\ast$-representation,
also denoted by $\overline{\pi_0}^{\rm ad}$, of the amalgamated free product C$^\ast$-algebra
\begin{equation}\label{amalgamated}
C^\ast(\G)\ast_{C^\ast(K)} \big[C^\ast(K\rtimes L^\infty(K,\mu))\big]
\end{equation}
into $B(\overline {H_0}^{\rm ad})$.
It is proved in \cite{Ra5} that  the character of the representation $\overline{\pi_0}^{\rm ad}$  is determined completely by the character of the representation $\pi_0$.

In general, if $L$ is a Hilbert subspace such that $$H_0\cong L\otimes \ell^2(\Gamma)$$
\noindent  and $$\pi_0|_\Gamma\cong {\rm \ Id}\otimes \lambda_\Gamma,$$ then $L$ is identified with the subspace of $\Gamma$-invariant vectors. For $\Gamma_0\in \S$, the space
$H^{\Gamma_0}_0$ of $\Gamma_0$-invariant vectors is
\begin{equation}\label{hzero}
H^{\Gamma_0}_0\cong L\otimes \ell^2(\Gamma_0\backslash \Gamma).
\end{equation}
\noindent  The space $\overline{H_0}^{\rm ad}$ is then defined by
$$\overline{H_0}^{\rm ad}\cong L\otimes L^2(K,\mu).$$
The factor $L^\infty(K,\mu)$ in formula (\ref{amalgamated}) acts in this representation by multiplication on the factor $L^2(K,\mu)$ in formula (\ref{hzero}).

In the case presented in this work, for a unitary representation $\pi$ as in the statement of Theorem \ref{split},  the unitary representation $\pi_0$ is Ad($\pi(g))_{g\in G}$
acting on the space of Hilbert-Schmidt operators on the Hilbert space $H_\pi$. Equivalently, $\pi_0$ is the unitary representation $\pi\otimes \overline{\pi}$ acting on
$H_\pi\otimes\overline {H_{\pi}}$. The hypothesis that $D_\pi=1$ corresponds to the fact that
$$H_0=H_{\pi}\otimes \overline {H_{\pi}}\cong \ell^2(\Gamma)\otimes \ell^2(\Gamma).$$
Then
$$\overline{H_0}^{\rm ad}\cong \ell^2(\Gamma)\otimes L^2(K,\mu).$$
\noindent Clearly, since
$$L^2(P,\tau)=L^2(\mathcal L(\Gamma\rtimes L^\infty(K,\mu)))\cong \ell^2(\Gamma)\otimes L^2(K,\mu),$$
\noindent  we obtain (using the algebras introduced in formula (\ref{asigma})), that
$$H^{\Gamma_\sigma}_0\cong L^2(\mathcal A_{\sigma \Gamma},\tau)\cong \ell^2(\Gamma)\otimes \ell^2(\Gamma_\sigma\backslash \Gamma),\quad \sigma \in G,$$
\noindent and
$$\overline{H_0}^{\rm ad}\cong\ell^2(\Gamma)\otimes L^2(K,\mu)\cong L^2(P,\tau).$$
The homomorphism $\alpha: G \rightarrow \operatorname{Aut} (P)$ which,  using conditional expectation on the algebras $\mathcal A_{\sigma \Gamma}, \sigma\in G$ ,
extends to C$^\ast(\G)$,  yields a  unitary representation $U_\alpha$ of $\G$ into
$L^2(P,\tau)$.

Then,  using the above isomorphism the representation $U_\alpha$ is unitarily equivalent to the unitary representation $\overline {\pi_0}^{\rm ad}$.
 \end{rem}

 We use the context from the Remark \ref{relation1} to state the following conjecture:

\begin{rem}
The statement of the  Ramanujan-Petersson conjectures may be generalized (\cite {Ra1}) to the question  regarding weak containment of the unitary representation
$\overline {\pi_0}^{\rm ad}$ of $G$ (equivalently of $\G$) in the left regular representation of $\G$ (restricted to $G$ if we analyze only representations of $G$).
Thus an equivalent form of the above conjecture is the question whether the unitary representation $U_\alpha|_{L^2(P)\ominus \mathbb C 1}$ is weakly contained in the
restriction of the left regular representation of $\G$ to $G$.

Recall, as observed in Theorem \ref{liftout}, that factorizing the homomorphism $\alpha$ to the group
${\rm Out}(P)={\rm Aut}(P)/{\rm Int}(P)$ one obtains a canonical homomorphism $G \rightarrow {\rm Out}(P)$ that is independent of choices.
We also recall that  the  canonical Cartan subalgebra  in the II$_1$ factor $\mathcal L(\Gamma\rtimes L^\infty(K,\mu))$ is unique, up to inner automorphisms (\cite{Op}, see also \cite{io}, \cite{pv}).

We conjecture that the property of weak containment in the left regular representation introduced above is independent of a cocycle perturbation
of the representation $U_\alpha$, as in Theorem  \ref{unicity}.

\end{rem}

{\bf Acknowledgment}. The author is deeply indebted to Ryszard Nest, Nicolas Monod and Stephan Vaes  for several comments on this paper.
The author is  deeply indebted to Florin Boca for several comments and for help with reorganizing the exposition of the paper.
The author is deeply indebted to Narutaka Ozawa for putting out his notes (\cite{Ra3}) on the paper \cite{Ra1},  and to allow their publication.

\end{document}